\documentclass{article}

\usepackage{color}

\usepackage{latexsym}
\usepackage{amssymb}
\usepackage{amsthm}
\usepackage{amssymb,amsmath}
\usepackage{MnSymbol}
\usepackage{mathrsfs}
\usepackage{dsfont}

\newtheorem{thm}{Theorem}
\newtheorem{lem}[thm]{Lemma}
\newtheorem{cor}[thm]{Corollary}
\newtheorem{prop}[thm]{Proposition}

\theoremstyle{definition}
\newtheorem{rem}[thm]{Remark}
\newtheorem{nota}[thm]{Notation}
\newtheorem{defi}[thm]{Definition}
\newtheorem{defi+nota}[thm]{Definition and Notation}

\title{Critical 3-hypergraphs\\(detailed version)}

\author{Abderrahim Boussa\"{\i}ri\footnotemark[1] \footnotemark[4]
\and 
Brahim Chergui \footnotemark[1] \footnotemark[3]
\and 
Pierre Ille\footnotemark[2] \footnotemark[5]
\and Mohamed Zaidi\footnotemark[1] \footnotemark[6]
}

\begin{document}

\maketitle

\footnotetext[1]{Facult\'e des Sciences A\"{\i}n Chock, 
D\'epartement de Math\'ematiques et Informatique, Km 8 route d'El Jadida, 
BP 5366 Maarif, Casablanca, Maroc}
\footnotetext[2]{Aix Marseille Univ, CNRS, Centrale Marseille, I2M, Marseille, France}
\footnotetext[4]{{\tt aboussairi@hotmail.com}}
\footnotetext[3]{{\tt cherguibrahim@gmail.com}}
\footnotetext[5]{{\tt pierre.ille@univ-amu.fr}}
\footnotetext[6]{{\tt zaidi.fsac@gmail.com}}

\begin{abstract}
Given a 3-hypergraph $H$, 
a subset $M$ of $V(H)$ is a module of $H$ if for each $e\in E(H)$ such that 
$e\cap M\neq\emptyset$ and 
$e\setminus M\neq\emptyset$, there exists $m\in M$ such that $e\cap M=\{m\}$ and for every $n\in M$, we have 
$(e\setminus\{m\})\cup\{n\}\in E(H)$. 
For example, 
$\emptyset$, $V(H)$ and $\{v\}$, where $v\in V(H)$, are modules of $H$, called trivial. 
A 3-hypergraph is prime if all its modules are trivial. 
Furthermore, a prime 3-hypergraph is critical if all its induced subhypergraphs, obtained by removing one vertex, are not prime. 
We characterize the critical 3-hypergraphs. 
\end{abstract}

\medskip


\medskip


\section{Introduction}

Let $H$ be a 3-hypergraph. 
A tournament $T$, with the same vertex set as $H$, is a realization of $H$ if the edges of $H$ are exactly the 3-element subsets of the vertex set of $T$ that induce 3-cycles. 
In~\cite{BCIZ20}, we characterized the 3-hypergraphs that admit realizations (see \cite[Problem~1]{BILT04}). 
To obtain our characterization, we introduced a new notion of a module for hypergraphs. 
By using the modular decomposition tree, we demonstrated that a 3-hypergraph is realizable if and only if all its prime (in terms of modular decomposition) induced subhypergraphs are realizable (see \cite[Theorem~13]{BCIZ20}). 
Moreover, given a realizable 3-hypergraph $H$, we proved that $H$ is prime if and only if its realizations are prime (see \cite[Theorem~12]{BCIZ20}). 
These results lead us to study the prime and induced subhypergraphs of a 
prime 3-hypergraph. 
Precisely, consider a prime 3-hypergraph $H$. 
In \cite[Theorem~13]{BCIZ}, we proved that $H$ admits a prime induced subhypergraph obtained by removing 1 or 2 vertices. 
Similar results were obtained for prime digraphs \cite{I97}, for 
prime binary relational structures \cite{ST93}, or for prime 2-structures \cite{EHR99}. 
Our purpose is to characterize the critical 3-hypergraphs, that is, the prime 3-hypergraphs all the subhypergraphs of which, obtained by removing one vertex, are not prime. 

At present, we formalize our presentation. 
We consider only finite structures. 
A hypergraph $H$ is defined by a vertex set $V(H)$ and an edge set $E(H)$, where 
$E(H)\subseteq 2^{V(H)}\setminus\{\emptyset\}$. 
Given $k\geq 2$, a hypergraph $H$ is a {\em $k$-hypergraph} if 
\begin{equation*}
E(H)\subseteq\binom{V(H)}{k}. 
\end{equation*}
Furthermore, a hypergraph $H$ is a {\em $\{k,k+1\}$-hypergraph} if 
\begin{equation*}
E(H)\subseteq\binom{V(H)}{k}\cup\binom{V(H)}{k+1}. 
\end{equation*}
Let $H$ be a hypergraph. 
With each $W\subseteq V(H)$, we associate the {\em subhypergraph} $H[W]$ of $H$ induced by $W$, which is defined on $V(H[W])=W$ by $E(H[W])=\{e\in E(H):e\subseteq W\}$. 

\begin{defi}\label{defi_module_hyper}
Let $H$ be a hypergraph. 
A subset $M$ of $V(H)$ is a {\em module} of $H$ if for each $e\in E(H)$ such that $e\cap M\neq\emptyset$ and 
$e\setminus M\neq\emptyset$, there exists $m\in M$ such that $e\cap M=\{m\}$, and for every $n\in M$, we have 
$$(e\setminus\{m\})\cup\{n\}\in E(H).$$ 
\end{defi}

Let $H$ be a hypergraph. 
Clearly, $\emptyset$, $V(H)$ and $\{v\}$, where $v\in V(H)$, are modules of $H$, called {\em trivial modules}. 
A hypergraph $H$ is {\em indecomposable} if all its modules are trivial, otherwise it is {\em decomposable}. 
A hypergraph $H$ is {\em prime} if it is indecomposable, with $v(H)\geq 3$. 
In \cite{BCIZ}, we prove the following result (see \cite[Theorem~9]{BCIZ}). 
We need the following notation (see \cite[Notation~5]{BCIZ}). 

\begin{nota}\label{nota_p_hyper}
Let $H$ be a hypergraph. 
Given $X\subsetneq V(H)$ such that $H[X]$ is prime, consider the following
subsets of $V(H)\setminus X$ 
\begin{itemize}
\item ${\rm Ext}_H(X)$ denotes the set of $v\in V(H)\setminus X$ such that 
$H[X\cup\lbrace v\rbrace ]$ is prime; 
\item $\langle X\rangle_H$ denotes the set of $v\in V(H)\setminus X$ such that 
$X$ is a module of $H[X\cup\lbrace v\rbrace ]$; 
\item for each $y\in X$, $X_H(y)$ denotes the set of $v\in V(H)\setminus X$ such that 
$\{y,v\}$ is a module of $H[X\cup\lbrace v\rbrace ]$.
\end{itemize}
The set $\{{\rm Ext}_H(X),\langle X\rangle_H\}\cup\{X_H(y):y\in X\}$ is denoted by $p_{(H,X)}$. 
\end{nota}

\begin{thm}\label{prop_Eh-2}
Let $H$ be a 3-hypergraph. 
Consider $X\subsetneq V(H)$ such that $H[X]$ is prime. 
Set $$\underline{X}=\{y\in X:X_H(y)\neq\emptyset\}.$$
If $H$ is prime, then there exist $v,w\in (V(H)\setminus X)\cup\underline{X}$ such that 
$H-\{v,w\}$ is prime. 
\end{thm}

The next result follows from Theorem~\ref{prop_Eh-2} (see \cite[Corollary~10]{BCIZ})

\begin{cor}\label{thm1_prime}
Let $H$ be a prime 3-hypergraph. 
If $v(H)\geq 4$, then there exist $v,w\in V(H)$ such that $H-\{v,w\}$ is prime. 
\end{cor}

Lastly, a prime hypergraph $H$ is {\em critical} if $H-v$ is decomposable for each $v\in V(H)$. 
Our purpose is to characterize the critical 3-hypergraphs. 

Corollary~\ref{thm1_prime} leads Ille~\cite{I93} to introduce the following auxiliary graph. 

\begin{defi}\label{defi_primality_graph}
Let $H$ be a prime 3-hypergraph with $v(H)\geq 5$. 
The {\em primality graph} $\mathscr{P}(H)$ associated with $H$ is the graph defined on $V(H)$ as follows. 
Given distinct $v,w\in V(H)$, $vw\in E(\mathscr{P}(H))$ if $H-\{v,w\}$ is prime. 
When $H$ is critical, it follows from Corollary~\ref{thm1_prime} that $\mathscr{P}(H)$ is nonempty. 
\end{defi}

\subsection{Critical and realizable 3-hypergraphs}

Let $T$ be a tournament. 
A subset $M$ of $V(T)$ is a {\em module} \cite{S92} of $T$ provided that for any $x,y\in M$ and $v\in V(T)$, 
if $xv,vy\in A(T)$, then $v\in M$. 
Note that the notions of a module and of a convex subset \cite{HW96} coincide for tournaments. 
Moreover, note that the notions of a module and of an interval coincide for linear orders. 
Given a tournament $T$, $\emptyset$, $V(T)$ and $\{v\}$, where $v\in V(T)$, are modules of $T$, called {\em trivial modules}. 
A tournament is {\em indecomposable} if all its modules are trivial, otherwise it is {\em decomposable}. 
A tournament $T$ is {\em prime} if it is indecomposable, with $v(T)\geq 3$.  
Lastly, a prime tournament $T$ is {\em critical} if $T-v$ is decomposable for each $v\in V(T)$. 
Schmerl and Trotter \cite{ST93} characterized the
critical tournaments. 
They obtained the tournaments $T_{2n+1}$, $U_{2n+1}$ and $W_{2n+1}$
defined on $\{0,\ldots, 2n\}$, where $n\geq 1$, as follows. 
\begin{itemize}
\item The tournament $T_{2n+1}$ is obtained from $L_{2n+1}$ by reversing all the arcs between even and odd vertices. 
\item The tournament $U_{2n+1}$ is obtained from $L_{2n+1}$ by reversing all the arcs between even vertices. 
\item The tournament $W_{2n+1}$ is obtained from $L_{2n+1}$ by reversing all the arcs between
$2n$ and the even elements of $\{0,\ldots,2n-1\}$. 
\end{itemize}

\begin{thm}[Schmerl and Trotter \cite{ST93}]\label{thm _Critical_ST}
Given a tournament $\tau$, with $v(\tau)\geq 5$, 
$\tau$ is critical if and only if $v(\tau)$ is odd and $\tau$ is isomorphic to 
$T_{v(\tau)}$, $U_{v(\tau)}$, or $W_{v(\tau)}$. 
\end{thm}

A realization of a 3-hypergraph is defined as follows (see ~\cite[Definition~10]{BCIZ20}). 
To begin, we associate with each tournament a 3-hypergraph in the following way 
(see ~\cite[Definition~9]{BCIZ20}). 

\begin{defi}\label{C_3}
The {\em $3$-cycle} is the tournament $C_3=(\{0,1,2\},\{01,12,20\})$. 
Given a tournament $T$, the {\em $C_3$-structure} of $T$ is the 3-hypergraph $C_3(T)$ defined on 
$V(C_3(T))=V(T)$ by 
$$E(C_3(T))=\{X\subseteq V(T):T[X]\ \text{is isomorphic to}\ C_3\}.$$
\end{defi}

\begin{defi}
Given a 3-uniform hypergraph $H$, a tournament $T$, with $V(T)=V(H)$, {\em realizes} $H$ if $H=C_3(T)$. 
We say also that $T$ is a {\em realization} of $H$. 
\end{defi}

In~\cite{BCIZ20}, we proved the following result (see  \cite[Theorem~12]{BCIZ20})

\begin{thm}\label{thm_same_primality}
Consider a realizable and 3-hypergraph $H$. 
For a realization $T$ of $H$, we have $H$ is prime if and only if $T$ is prime.
\end{thm}

The next result follows from Theorems~\ref{thm _Critical_ST} and \ref{thm_same_primality} (see \cite[Theorem~49]{BCIZ20}). 

\begin{cor}\label{cor_reali_critical}
Given a realizable 3-hypergraph $H$, 
$H$ is critical if and only if $v(H)$ is odd and $H$ is isomorphic to $C_3(T_{v(H)})$, $C_3(T_{v(H)})$, or 
$C_3(T_{v(H)})$. 
\end{cor}

\begin{defi}\label{defi_reali_critical}
Given a critical 3-hypergraph $H$, we say that $H$ is {\em circular} if $v(H)$ is odd and 
$H$ is isomorphic to $C_3(T_{v(H)})$. 
\end{defi}

\subsection{A construction of 3-hypergraphs}

The construction is done from the set of the components of a graph. 
We use the following notation. 

\begin{nota}\label{nota_comp}
Consider a graph $\Gamma$. 
The set of the components of $\Gamma$ is denoted by $\mathfrak{C}(\Gamma)$. 
Furthermore, set 
$\mathfrak{C}_{{\rm even}}(\Gamma)=\{C\in\mathfrak{C}(\Gamma):v(C)\equiv 0 \mod 2\}$ and 
$\mathfrak{C}_{{\rm odd}}(\Gamma)=\{C\in\mathfrak{C}(\Gamma):v(C)\equiv 1 \mod 2\}$. 
Lastly, set 
$\mathfrak{C}_1(\Gamma)=\{C\in\mathfrak{C}(\Gamma):v(C)=1\}$ and 
$V_1(\Gamma)=\{v\in V(\Gamma):\{v\}\in\mathfrak{C}_1(\Gamma)\}$. 
For each $C\in\mathfrak{C}(\Gamma)$, set 
\begin{equation*}
w(C)=\ 
\begin{cases}
\text{$v(C)/2$ if $C\in\mathfrak{C}_{{\rm even}}(\Gamma)$}\\
\text{or}\\
\text{$(v(C)-1)/2$ if $C\in\mathfrak{C}_{{\rm odd}}(\Gamma)$}.
\end{cases}
\end{equation*}
\end{nota}

We consider a graph $\Gamma$, 
\begin{equation}\label{critical_construction_0}
\text{all the components of which are paths.} 
\end{equation}
For each $n\geq 1$, recall that the {\em path} $P_n$ is defined on $V(P_n)=\{0,\ldots,n-1\}$ by 
\begin{equation*}
E(P_n)=\ 
\begin{cases}
\text{$\emptyset$ if $n=1$,}\\
\text{$\{k(k+1):0\leq k\leq n-2\}$ if $n\geq 2$.}
\end{cases}
\end{equation*}
We suppose that 
\begin{align}\label{critical_construction_1}
\filledstar\ &\mathfrak{C}(\Gamma)\setminus\mathfrak{C}_1(\Gamma)\neq\emptyset,&\nonumber\\
\filledstar\ &\text{for each $C\in\mathfrak{C}_{{\rm odd}}(\Gamma)$, if $V(H)\setminus V(C)\neq\emptyset$, then}\nonumber\\
&\hspace{40mm}\text{
$(\mathfrak{C}(\Gamma)\setminus\mathfrak{C}_1(\Gamma))\setminus\{C\}\neq\emptyset$,}\\
\filledstar\ &\text{for each $C\in\mathfrak{C}_{{\rm even}}(\Gamma)$, 
if $|V(H)\setminus V(C)|\geq 2$, then}\nonumber\\
&\hspace{40mm}\text{
$(\mathfrak{C}(\Gamma)\setminus\mathfrak{C}_1(\Gamma))\setminus\{C\}\neq\emptyset$}.\nonumber
\end{align}
Consider also a $\{2,3\}$-hypergraph $\mathbb{H}$ defined on $\mathfrak{C}(\Gamma)$ satisfying 
\begin{align}\label{critical_construction_2}
\filledsquare\ &\text{for each $\varepsilon\in E(\mathbb{H})$, $|\varepsilon|=2$ or $3$},\nonumber\\
\filledsquare\ &\text{for each $\varepsilon\in E(\mathbb{H})$, if $|\varepsilon|=2$, then 
$\varepsilon\cap\mathfrak{C}_{{\rm even}}(\Gamma)\neq\emptyset$ and 
$\varepsilon\cap\mathfrak{C}_{{\rm odd}}(\Gamma)\neq\emptyset$,}\\
\filledsquare\ &\text{for each $\varepsilon\in E(\mathbb{H})$, if $|\varepsilon|=3$, then 
$\varepsilon\subseteq\mathfrak{C}_{{\rm odd}}(\Gamma)$.\nonumber}
\end{align}
With $\Gamma$ and $\mathbb{H}$, we associate the 3-hypergraph 
$\Gamma\bullet\mathbb{H}$ defined on $V(\Gamma)$ in the following manner. 
For each $C\in\mathfrak{C}(\Gamma)$, 
we consider an isomorphism $\varphi_C$ from the path $P_{v(C)}$ onto $C$. 
With each $C\in\mathfrak{C}_{{\rm odd}}(\Gamma)\setminus\mathfrak{C}_{1}(\Gamma)$, 
we associate the set 
$$E_C=\varphi_C(E(C_3(U_{v(C)}))).$$
Moreover, consider $\varepsilon\in E(\mathbb{H})$ such that $|\varepsilon|=2$. 
There exist $C\in\mathfrak{C}_{{\rm even}}(\Gamma)$ and 
$D\in\mathfrak{C}_{{\rm odd}}(\Gamma)$ such that $e=CD$. 
Associate with $\varepsilon$ the set 
$$E_\varepsilon=\{\varphi_C(2i)\varphi_C(2j+1)\varphi_D(2k):\ 0\leq i\leq j\leq w(C)-1,
 \ 0\leq k\leq w(D)\}.$$
Lastly, consider $\varepsilon\in E(\mathbb{H})$ such that $|\varepsilon|=3$. 
There exist $I,J,K\in\mathfrak{C}_{{\rm odd}}(\Gamma)$ such that $\varepsilon=IJK$. 
Associate with $\varepsilon$ the set 
$$E_\varepsilon=
\{\varphi_{I}(2i)\varphi_{J}(2j)\varphi_{K}(2k):0\leq i\leq w(I),
\ 0\leq j\leq w(J),
\ 0\leq k\leq w(K)\}.$$
The 3-hypergraph $\Gamma\bullet\mathbb{H}$ is defined on $V(\Gamma)$ by 
\begin{equation}\label{E1_Rbis_Critical_hyper}
E(\Gamma\bullet\mathbb{H})=(\bigcup_{C\in\mathfrak{C}_{{\rm odd}}(\Gamma)\setminus\mathfrak{C}_{1}(\Gamma)}E_C)\cup
(\bigcup_{\varepsilon\in E(\mathbb{H})}E_\varepsilon).
\end{equation}

\subsection{The main results}

To begin, we prove the following theorem. 

\begin{thm}\label{thm_1_main}
Let $H$ be a critical and 3-hypergraph such that $v(H)\geq 5$. 
If $H$ is not circular, then there exist a graph $\Gamma$ satisfying \eqref{critical_construction_0} and \eqref{critical_construction_1}, and a 
$\{2,3\}$-hypergraph $\mathbb{H}$ defined on $\mathfrak{C}(\Gamma)$ satisfying 
\eqref{critical_construction_1} such that $H=\Gamma\bullet\mathbb{H}$.
\end{thm}

In Theorem~\ref{thm_1_main}, the graph $\Gamma$ is the primality graph of $H$ 
(see Definition~\ref{defi_primality_graph}). 
To state, the second main theorem, we consider a graph $\Gamma$ satisfying \eqref{critical_construction_0} and \eqref{critical_construction_1}. 
We consider also a $\{2,3\}$-hypergraph $\mathbb{H}$ defined on $\mathfrak{C}(\Gamma)$ satisfying 
\eqref{critical_construction_2}. 
Furthermore, we use the modules of $\mathbb{H}$ defined as follows. 

\begin{defi}
Given $\mathbb{W}\subseteq V(\mathbb{H})$, $\mathbb{W}$ is a module of 
$\mathbb{H}$ if $\mathbb{W}$ is a module of 
$$(V(\mathbb{H}),E(\mathbb{H})\cap\binom{V(\mathbb{H})}{2})$$ and 
$\mathbb{W}$ is a module of 
$$(V(\mathbb{H}),E(\mathbb{H})\cap\binom{V(\mathbb{H})}{3}).$$
\end{defi}

\begin{thm}\label{thm_2_main}
Suppose that $v(\Gamma\bullet\mathbb{H})\geq 5$. 
The 3-hypergraph $\Gamma\bullet\mathbb{H}$ is critical if and only if the following three assertions hold
\begin{enumerate}
\item[(C1)] $\mathbb{H}$ is connected;
\item[(C2)] for every nontrivial module $\mathbb{M}$ of $\mathbb{H}$, we have 
$\mathbb{M}\setminus\mathfrak{C}_{1}(\Gamma)\neq\emptyset$;
\item[(C3)] for each $v\in V_1(\Gamma)$, if 
$\mathbb{H}-\{v\}$ is connected, then $\mathbb{H}-\{v\}$ admits a nontrivial module $\mathbb{M}_{\{v\}}$ such that 
$\mathbb{M}_{\{v\}}\subseteq\mathfrak{C}_1(\Gamma)\setminus\{\{v\}\}$.
\end{enumerate}
\end{thm}

Before proving Theorem~\ref{thm_2_main}, we establish the next proposition. 

\begin{prop}\label{Prop1_Critical_hyper}
The 3-hypergraph $\Gamma\bullet\mathbb{H}$ is prime if and only if the following two assertions hold
\begin{enumerate}
\item[(C1)] $\mathbb{H}$ is connected;
\item[(C2)] for every nontrivial module $\mathbb{M}$ of $\mathbb{H}$, we have 
$\mathbb{M}\setminus\mathfrak{C}_{1}(\Gamma)\neq\emptyset$. 
\end{enumerate}
\end{prop}

Finally, we improve Theorem~\ref{thm_2_main} by characterizing the 3-hypergraph $\Gamma\bullet\mathbb{H}$ whenever it is critical and satisfies 
$\mathscr{P}(\Gamma\bullet\mathbb{H})=\Gamma$ (see Theorem~\ref{thm_3_main} in Section~\ref{sec_improvement}). 

\section{The critical tournaments}

The primality graph associated with a prime tournament is defined in the same as that associated with a 
prime 3-hypergraph (see Definition~\ref{defi_primality_graph}). 

\begin{defi}\label{defi_primality_graph_tournament}
Let $T$ be a prime tournament. 
The {\em primality graph} $\mathscr{P}(T)$ of $T$ is defined on $V(T)$ as follows. 
Given distinct $v,w\in V(T)$, $vw\in E(\mathscr{P}(T))$ if $T-\{v,w\}$ is prime. 
\end{defi}

The basic properties of the primality graph follow. 
The next lemma is stated in ~\cite{I93} without a proof. For a proof, see \cite[Lemma~10]{BI09}. 

\begin{lem} [Ille~\cite{I93}]\label{lem Boud Ille}
Let $T$ be a critical tournament with $v(T)\geq 5$. 
For every $v\in\mathscr{C}(T)$, we have 
$d_{\mathscr{P}(T)}(v)\leq 2$.
Moreover, the next two assertions hold. 
\begin{enumerate}
\item Given $v\in V(T)$, if $d_{\mathscr{P}(T)}(v)=1$, then 
$V(T)\setminus(\{v\}\cup N_{\mathscr{P}(T)}(v))$ is a module of $T-v$. 
\item Given $v\in V(T)$, if $d_{\mathscr{P}(T)}(v)=2$, then 
$N_{\mathscr{P}(T)}(v)$ is a module of $T-v$. 
\end{enumerate}
\end{lem}

For each $n\geq 3$, recall that the {\em cycle} $C_n$ is defined on $V(C_n)=\{0,\ldots,n-1\}$ by 
$E(C_n)=E(P_n)\cup\{0(n-1)\}$. 
The {\em length} of $C_n$ is $n$. 
Given a critical tournament $T$, it follows from Lemma~\ref{lem Boud Ille} that the connected components of $\mathscr{P}(T)$ are paths or cycles. 
Boudabbous and Ille \cite{BI09} characterized the critical tournaments from their 
primality graphs. 
To begin, they examine the connected components of the primality graph associated with a critical tournament. 

\begin{lem}[Corollary 17 \cite{BI09}]\label{cor_BI09}
If $T$ is a critical tournament, with $v(T)\geq 5$, then $\mathscr{P}(T)$ satisfies one of the following 
\begin{enumerate}
\item $\mathscr{P}(T)$ is a cycle of odd length;
\item $\mathscr{P}(T)$ is a path; 
\item $v(T)$ is odd and there is $v\in V(T)$ such that $\mathscr{P}(T)-v$ is a path and 
$N_{\mathscr{P}(T)}(v)=\emptyset $. 
\end{enumerate}
\end{lem}

For each of the three shapes describe in Lemma~\ref{cor_BI09}, 
Boudabbous and Ille \cite{BI09} characterized the corresponding critical tournaments. 

\begin{prop}[Proposition 18 \cite{BI09}]\label{prop_T2n+1}
Given a tournament such that $v(T)\geq 5$, 
$T$ is critical and $\mathscr{P}(T)=C_{v(T)}$ if and only if $v(T)$ is odd and $T=T_{2n+1}$ or 
$(T_{2n+1})^\star$. 
\end{prop}

\begin{prop}[Proposition 19 \cite{BI09}]\label{prop_U2n+1}
Given a tournament such that $v(T)\geq 5$, 
$T$ is critical and $\mathscr{P}(T)=P_{v(T)}$ if and only if $v(T)$ is odd and $T=U_{v(T)}$ or 
$(U_{v(T)})^\star$. 
\end{prop}

\begin{prop}[Proposition 21 \cite{BI09}]\label{prop_W2n+1}
Given a tournament defined on $\{0,\ldots,2n\}$ with $2n+1\geq 5$, 
$T$ is critical, $\mathscr{P}(T)-2n=P_{2n-1}$ 
and $N_{\mathscr{P}(T)}(2n)=\emptyset$ if and only if $T=W_{2n+1}$ or 
$(W_{2n+1})^\star$. 
\end{prop}

\section{Proof of Theorem~\ref{thm_1_main}}

The purpose of this section is to characterize the non circular and critical 3-hypergraphs. 
We use the following notation. 

\begin{nota}\label{equiv}
Let $H$ be a 3-hypergraph. 
\begin{equation*}
\text{For}\ e,f\in\binom{V(H)}{3}, \ e\equiv_H f\ \text{means}\ 
\begin{cases}
e,f\in E(H)\\
\text{or}\\
e,f\not\in E(H).
\end{cases}
\end{equation*}
\end{nota}

An analogue of lemma~\ref{lem Boud Ille} for prime and 3-uniform hypergraphs follows. 
For a proof, see \cite[Lemma~10]{BI09}. 

\begin{lem}\label{Lcritical}
Let $H$ be a critical and 3-hypergraph with $v(H)\geq 5$. 
For every $v\in V(H)$, we have $d_{\mathscr{P}(H)}(v)\leq 2$. 
Moreover, the next two assertions hold. 
\begin{enumerate}
\item Given $v\in V(H)$, if $d_{\mathscr{P}(H)}(v)=1$, then 
$V(H)\setminus(\{v\}\cup N_{\mathscr{P}(H)}(v))$ is a module of $H-v$. 
\item Given $v\in V(H)$, if $d_{\mathscr{P}(H)}(v)=2$, then 
$N_{\mathscr{P}(H)}(v)$ is a module of $H-v$. 
\end{enumerate}
\end{lem}

Given a critical 3-hypergraph $H$, it follows from Lemma~\ref{Lcritical} that the components of 
$\mathscr{P}(H)$ are cycles or paths.

\begin{prop}\label{Pcritical_T2n+1}
Let $H$ be a critical 3-hypergraph defined on $\{0,\ldots,p-1\}$, where $p\geq 5$. 
If there exists $k\in\{3,\ldots,p\}$ such that $\mathscr{P}(H)[\{0,\ldots,k-1\}]=C_k$, then 
$p=2n+1$, $k=p$, and $H=C_3(T_{2n+1})$, where $n\geq 2$. 
\end{prop}

\begin{proof}
Suppose that there exists $k\in\{3,\ldots,p\}$ such that $\mathscr{P}(H)[\{0,\ldots,k-1\}]=C_k$. 
First, we show that 
\begin{equation}\label{E0_Pcritical}
\text{$k$ is odd.}
\end{equation} 
Otherwise, there exists $l\geq 2$ such that $\mathscr{P}(H)[\{0,\ldots,2l-1\}]=C_{2l}$. 
We verify that $\{0,2\}$ is a module of $H$. 
Consider $e\in E(H)$ such that $e\cap\{0,2\}\neq\emptyset$ and 
$e\setminus\{0,2\}\neq\emptyset$. 
Suppose for a contradiction that $0,2\in e$. 
There exists $i\in\{1\}\cup\{3,\ldots,2l-1\}$ such that $e=02i$. 
Since $N_{\mathscr{P}(H)}(1)=\{0,2\}$, it follows from Lemma~\ref{Lcritical} that $\{0,2\}$ is a module of $H-1$. 
It follows that $i=1$, that is, $e=012$. 
It follows from Lemma~\ref{Lcritical} that $01(2l-2)\in E(H)$, which contradicts the fact that 
$\{0,2l-2\}$ is a module of $H-(2l-1)$. 
Consequently, $|e\cap\{0,2\}|=1$. 
For a contradiction, suppose that $1\in e$. 
Since $|e\cap\{0,2\}|=1$, there $j\in\{3,\ldots,2l-1\}$ such that $e=01j$ or $12j$. 
Denote by $j'$ the unique element of  $\{2l-2,2l-1\}$ such that $j'\equiv j\mod 2$. 
It follows from Lemma~\ref{Lcritical} that 
$(e\setminus\{j\})\cup\{j'\}\in E(H)$. 
Similarly, by denoting by $j''$ the unique element of  $\{3,4\}$ such that $j''\equiv j\mod 2$, we obtain $(e\setminus\{j\})\cup\{j''\}\in E(H)$. 
We distinguish the following two cases. 
\begin{enumerate}
\item Suppose that $e=01j$. 
If $j$ is even, then $01(2l-2)\in E(H)$, which contradicts the fact that $\{0,2l-2\}$ is a module of $H-(2l-1)$. 
If $j$ is odd, then $013\in E(H)$, which contradicts the fact that $\{1,3\}$ is a module of $H-2$. 
\item Suppose that $e=12j$. 
If $j$ is even, then $124\in E(H)$, which contradicts the fact that $\{2,4\}$ is a module of $H-3$. 
If $j$ is odd, then $12(2l-1)\in E(H)$, which contradicts the fact that $\{1,2l-1\}$ is a module of $H-0$. 
\end{enumerate}
It follows that $1\not\in e$. So, $e\in E(H-1)$. 
Since $N_{\mathscr{P}(H)}(1)=\{0,2\}$, it follows from Lemma~\ref{Lcritical} that 
$\{0,2\}$ is a module of $H-1$. 
Thus, there exists $i\in\{0,2\}$ such that $e\cap\{0,2\}=\{i\}$, and 
$(e\setminus\{i\})\cup\{i'\}\in E(H)$ for each $i'\in\{0,2\}$. 
Consequently, $\{0,2\}$ is a module of $H$, which contradicts the fact that $H$ is prime. 
It follows that \eqref{E0_Pcritical} holds. 
Set $k=2n+1$, where $n\geq 1$. 

Second, we prove that $\{0,\ldots,2n\}$ is a module of $H$. 
Consider $e\in E(H)$ such that $e\cap \{0,\ldots,2n\}\neq\emptyset$ and 
$e\setminus \{0,\ldots,2n\}\neq\emptyset$. 
We prove that 
\begin{equation}\label{E2_Pcritical}
|e\cap \{0,\ldots,2n\}|=1.
\end{equation}
Otherwise, we have $|e\cap \{0,\ldots,2n\}|=2$. 
There exist $i,j\in\{0,\ldots,2n\}$, with $i<j$, such that $e\cap \{0,\ldots,2n\}=\{i,j\}$. 
Denote by $i'$ the unique element of $\{0,1\}$ such that $i'\equiv i\mod 2$. 
As previously, we obtain $(e\setminus\{i\})\cup\{i'\}\in E(H)$. 
Denote by $j'$ the unique element of $\{2,3\}$ such that $j'\equiv j\mod 2$. 
We obtain $(e\setminus\{i,j\})\cup\{i',j'\}\in E(H)$. 
Set $e'=(e\setminus\{i,j\})\cup\{i',j'\}$. 
Observe that $e'\setminus\{0,\ldots,2n\}=e\setminus\{0,\ldots,2n\}$, and denote by $v$ the unique element of $e'\setminus\{0,\ldots,2n\}$. 
If $i'=0$ and $j'=2$, then $02v\in E(H)$, which contradicts the fact that $\{0,2\}$ is a module of $H-1$. 
If $i'=1$ and $j'=3$, then $13v\in E(H)$, which contradicts the fact that $\{1,3\}$ is a module of $H-2$. 
Suppose that $i'=0$ and $j'=3$. 
We get $03v\in E(H)$, and hence $0(2n-1)v\in E(H)$, which contradicts the fact that $\{0,2n-1\}$ is a module of $H-(2n)$. 
Lastly, if $i'=1$ and $j'=4$, then $1(2n)v\in E(H)$, which contradicts the fact that $\{1,2n\}$ is a module of $H-0$. 
It follows that \eqref{E2_Pcritical} holds. 
Denote by $i$ the unique element of $e\cap\{0,\ldots,2n\}$. 
For every $l\in\{1,\ldots,n\}$, we obtain 
\begin{equation}\label{E3_Pcritical}
(e\setminus\{i\})\cup\{i+2l\}\in E(H),
\end{equation}
where $i+2l$ is considered modulo $2n+1$. 
In particular, we have $(e\setminus\{i\})\cup\{i+2n\}\in E(H)$, that is, 
$(e\setminus\{i\})\cup\{i-1\}\in E(H)$. 
Since $\{i-1,i+1\}$ is a module of $H-i$, we get $(e\setminus\{i\})\cup\{i+1\}\in E(H)$. 
For each $m\in\{0,\ldots,n-1\}$, we obtain 
\begin{equation}\label{E4_Pcritical}
(e\setminus\{i\})\cup\{i+2m+1\}\in E(H).
\end{equation}
It follows from \eqref{E3_Pcritical} and \eqref{E4_Pcritical} that 
$(e\setminus\{i\})\cup\{i'\}\in E(H)$ for every $i'\in\{0,\ldots,2n\}$. 
Consequently, $\{0,\ldots,2n\}$ is a module of $H$. 
Since $H$ is prime, we obtain $V(H)=\{0,\ldots,2n\}$. 

Third, we prove that $H=C_3(T_{2n+1})$. 
We have 
\begin{align}\label{E5_Pcritical}
E(C_3(T_{2n+1}))=&\{(2i)(2l+1)(2j):0\leq i\leq l<j\leq n\}\nonumber\\
&\cup\{(2i+1)(2l)(2j+1):0\leq i<p\leq l\leq n-1\}.
\end{align}
For a contradiction, suppose that there are $x,y,z\in\{0,\ldots,2n\}$ such that 
$x<y<z$, $x\equiv y\mod 2$, and $xyz\in E(H)$. 
It follows from Lemma~\ref{Lcritical} that 
$(y-2)yz\in E(H)$, which contradicts the fact that $\{y-2,y\}$ is a module of $H-(y-1)$. 
Hence, 
$xyz\not\in E(H)$ for $x,y,z\in\{0,\ldots,2n\}$ such that 
$x<y<z$ and $x\equiv y\mod 2$. 
Similarly, 
$xyz\not\in E(H)$ for $x,y,z\in\{0,\ldots,2n\}$ such that 
$x<y<z$ and $y\equiv z\mod 2$. 
It follows that 
\begin{equation}\label{E8_Pcritical}
E(H)\subseteq E(C_3(T)).
\end{equation}
Now, consider $x,y,z\in\{0,\ldots,2n\}$ such that 
$x<y<z$, $x\not\equiv y\mod 2$, and $y\not\equiv z\mod 2$. 
It follows from Lemma~\ref{Lcritical} that 
\begin{equation}\label{E9_Pcritical}
xyz\equiv_Hx(x+1)(x+2).
\end{equation}
Now, we prove that the permutation 
\begin{equation*}
\begin{array}{rrcl}
\theta:&\{0,\ldots,2n\}&\longrightarrow&\{0,\ldots,2n\}\\
&x&\longmapsto&x+1\mod 2n+1
\end{array}
\end{equation*}
of $\{0,\ldots,2n\}$ is an automorphism of $H$. 
Given $x,y,z\in\{0,\ldots,2n\}$ such that $x<y<z$, we have to verify that 
$xyz\equiv_H(x+1)(y+1)(z+1)$. 
We have $$xyz\equiv_Hxy(z+2)\equiv_Hx(y+2)(z+2)\equiv_H(x+2)(y+2)(z+2).$$
Thus, 
$xyz\equiv_H(x+2n+2)(y+2n+2)(z+2n+2)$, that is, 
$$xyz\equiv_H(x+1)(y+1)(z+1).$$
Consequently, $\theta$ is an automorphism of $H$. 
Since $H$ is prime, there exist 
$u,v,w\in\{0,\ldots,2n\}$ such that $u<v<w$ and 
$uvw\in E(H)$. 
By \eqref{E8_Pcritical}, $uvw\in E(C_3(T))$. 
It follows from \eqref{E5_Pcritical} that $u\not\equiv v\mod 2$ and 
$v\not\equiv w\mod 2$. 
It follows from \eqref{E9_Pcritical} that 
$u(u+1)(u+2)\in E(H)$. 
Since $\theta$ is an automorphism of $H$, we obtain 
\begin{equation}\label{E10_Pcritical}
012\in E(H).
\end{equation}
Finally, consider $x,y,z\in\{0,\ldots,2n\}$ such that $x<y<z$ and $xyz\in E(C_3(T))$. 
By \eqref{E5_Pcritical}, 
$x\not\equiv y\mod 2$ and $y\not\equiv z\mod 2$. 
It follows from \eqref{E9_Pcritical} that $xyz\equiv_Hx(x+1)(x+2)$. 
Since $\theta$ is an automorphism of $H$, we obtain $x(x+1)(x+2)\equiv_H012$. 
Therefore, $xyz\equiv_H012$. 
By \eqref{E10_Pcritical}, $xyz\in E(H)$. 
Consequently, we obtain $E(C_3(T))\subseteq E(H)$. 
It follows from \eqref{E8_Pcritical} that $H=C_3(T)$. 
\end{proof}

The next corollary is an easy consequence of Lemma~\ref{Lcritical} and Proposition~\ref{Pcritical_T2n+1}. 

\begin{cor}\label{C_critical_T2n+1}
Given a critical 3-hypergraph $H$, with $v(H)\geq 5$, 
$H$ is not circular if and only if all the components of $\mathscr{P}(H)$ are paths. 
\end{cor}

\begin{proof}
Suppose that $H$ is circular. 
By Definition~\ref{defi_reali_critical}, $H$ is isomorphic to $C_3(T_{v(H)})$. 
Hence, $\mathscr{P}(H)$ is isomorphic to $\mathscr{P}(C_3(T_{v(H)}))$. 
It follows from Theorem~\ref{thm_same_primality} that 
$$\mathscr{P}(C_3(T_{v(H)}))=\mathscr{P}(T_{v(H)}).$$
By Proposition~\ref{prop_T2n+1}, $\mathscr{P}(T_{v(H)})=C_{v(H)}$. 
It follows that $\mathscr{P}(H)$ contains a cycle among its components. 

Conversely, suppose that $\mathscr{P}(H)$ contains a cycle among its components. 
Up to isomorphy, we can assume that $V(H)=\{0,\ldots,p-1\}$, where $p\geq 5$, and there 
exists $k\in\{3,\ldots,p\}$ satisfying $\mathscr{P}(H)[\{0,\ldots,k-1\}]=C_k$. 
By Proposition~\ref{Pcritical_T2n+1}, $H=C_3(T_{2n+1})$. 
Therefore, $H$ is circular. 
\end{proof}

\begin{nota}\label{nota_non_circular}
Let $H$ be a non circular and critical 3-hypergraph. 
For each $C\in\mathfrak{C}(\mathscr{P}(H))\setminus\mathfrak{C}_1(\mathscr{P}(H))$, $C$ is a path. 
Thus, there exists an isomorphism $\varphi_C$ from $P_{v(C)}$ onto $C$. 
\end{nota}

\begin{lem}\label{Lcritical_U2n+1_V2n+1}
Let $H$ be a critical and 3-hypergraph with $v(H)\geq 5$. 
Suppose that $H$ is not circular. 
\begin{enumerate}
\item Let $C\in\mathfrak{C}_{{\rm odd}}(\mathscr{P}(H))\setminus\mathfrak{C}_1(\mathscr{P}(H))$. 
For every $e\in E(H)$, if $e\cap V(C)\neq\emptyset$ and $e\setminus V(C)\neq\emptyset$, then there exists 
$i\in\{0,\ldots,w(C)\}$ (see Notation~\ref{nota_comp}) such that $e\cap V(C)=\{\varphi_C(2i)\}$. 
\item Let $C\in\mathfrak{C}_{{\rm even}}(\mathscr{P}(H))$. 
For every $e\in E(H)$, if $e\cap V(C)\neq\emptyset$ and $e\setminus V(C)\neq\emptyset$, then there exist 
$i,j\in\{0,\ldots,w(C)-1\}$ such that $i\leq j$ and 
$e\cap V(C)=\{\varphi_C(2i),\varphi_C(2j+1)\}$. 
\end{enumerate}
\end{lem}

\begin{proof}
For the first assertion, consider $C\in\mathfrak{C}_{{\rm odd}}(\mathscr{P}(H))\setminus\mathfrak{C}_1(\mathscr{P}(H))$. 
Let $e\in E(H)$ such that $e\cap V(C)\neq\emptyset$ and 
$e\setminus V(C)\neq\emptyset$. 
We prove that 
\begin{equation}\label{E0_Pcritical_U2n+1}
e\cap\{\varphi_C(2i+1):i\in\{0,\ldots,w(C)-1\}\}=\emptyset.
\end{equation}
Otherwise, there exists $i\in\{0,\ldots,w(C)-1\}$ such that 
$\varphi_C(2i+1)\in e$. 
We distinguish the following two cases. 
\begin{itemize}
\item Suppose that there exists $j\in\{0,\ldots,i\}$ such that $\varphi_C(2j)\in e$. 
Since $e\setminus V(C)\neq\emptyset$, $e\cap V(C)=\{\varphi_C(2j),\varphi_C(2i+1)\}$. 
By denoting by $v$ the unique element of $e\setminus V(C)$, we obtain 
$e=\varphi_C(2j)\varphi_C(2i+1)v$. 
It follows from Lemma~\ref{Lcritical} that $\varphi_C(2j)\varphi_C(2w(C)-1)v\in E(H)$, which contradicts the fact that $V(H)\setminus\{\varphi_C(2w(C)-1),\varphi_C(2w(C))\}$ is a module of $H-\varphi_C(2w(C))$. 
\item Suppose that $e\cap\{\varphi_C(2j):j\in\{0,\ldots,i\}\}=\emptyset$. 
It follows from Lemma~\ref{Lcritical} that there exists $f\in E(H)$ such that $\varphi_C(1)\in f$, 
$\varphi_C(0)\not\in f$, and $f\setminus V(C)\neq\emptyset$, which contradicts the fact that 
$V(H)\setminus\{\varphi_C(0),\varphi_C(1)\}$ is a module of 
$H-\varphi_C(0)$. 
\end{itemize}
Consequently, \eqref{E0_Pcritical_U2n+1} holds. 

Now, we prove hat there exists $i\in\{0,\ldots,w(C)\}$ such that 
\begin{equation}\label{E3_Pcritical_U2n+1}
e\cap V(C)=\{\varphi_C(2i)\}.
\end{equation}
Otherwise, it follows from \eqref{E0_Pcritical_U2n+1} that there exist distinct $i,j\in\{0,\ldots,w(C)\}$ such that $0\leq i<j\leq w(C)$ such that $e\cap V(C)=\{\varphi_C(2i),\varphi_C(2j)\}$. 
It follows from Lemma~\ref{Lcritical} that there exists $f\in E(H)$ such that $f\cap V(C)=\{\varphi_C(2i),\varphi_C(2i+2)\}$ and $f\setminus V(C)\neq\emptyset$, which contradicts the fact that 
$\{\varphi_C(2i),\varphi_C(2i+2)\}$ is a module of 
$H-\varphi_C(2i+1)$. 
Consequently, \eqref{E3_Pcritical_U2n+1} holds. 

For the second assertion, consider $C\in\mathfrak{C}_{{\rm even}}(\mathscr{P}(H))$. 
Let $e\in E(H)$ such that $e\cap V(C)\neq\emptyset$ and 
$e\setminus V(C)\neq\emptyset$. 
Set $$p=\min(\{i\in\{0,\ldots,2w(C)-1\}:\varphi_C(i)\in e\}).$$
For a contradiction, suppose that $p$ is odd. 
It follows from Lemma~\ref{Lcritical} that $(e\setminus\{\varphi_C(p)\})\cup\{\varphi_C(1)\}\in E(H)$, which contradicts the fact that 
$V(C)\setminus\{\varphi_C(0),\varphi_C(1)\}$ is a module of $H-\varphi_C(0)$. 
It follows that $p$ is even. 
Thus, there exists $i\in\{0,\ldots,w(C)-1\}$ such that 
$p=\varphi_C(2i)$. 
Similarly, $\max(\{i\in\{0,\ldots,2w(C)-1\}:\varphi_C(i)\in e\})$ is odd. 
Thus there exists $j\in\{0,\ldots,w(C)-1\}$ such that 
$\max(\{i\in\{0,\ldots,2w(C)-1\}:\varphi_C(i)\in e\})=\varphi_C(2j+1)$. 
We obtain that $i\leq j$ and 
$e\cap V(C)=\{\varphi_C(2i),\varphi_C(2j+1)\}$. 
\end{proof}

\begin{prop}\label{Pcritical_U2n+1}
Let $H$ be a critical and 3-hypergraph with $v(H)\geq 5$. 
Suppose that $H$ is not circular. 
Consider $C\in\mathfrak{C}_{{\rm odd}}(\mathscr{P}(H))\setminus\mathfrak{C}_1(\mathscr{P}(H))$. 
\begin{enumerate}
\item Suppose that $V(C)\subsetneq V(H)$. There exists $e\in E(H)$, such that $e\cap V(C)\neq\emptyset$ and 
$e\setminus V(C)\neq\emptyset$. 
Moreover, one of the following two assertions holds 
\begin{itemize}
\item there exist distinct 
$D,D'\in\mathfrak{C}_{{\rm odd}}(\mathscr{P}(H))\setminus\{C\}$ such that 
$e\subseteq V(C)\cup V(D)\cup V(D')$, 
$e\cap V(D)\neq\emptyset$, and $e\cap V(D')\neq\emptyset$;
\item there exists 
$D\in\mathfrak{C}_{{\rm even}}(\mathscr{P}(H))\setminus\{C\}$ such that 
$e\subseteq V(C)\cup V(D)$, and 
$e\cap V(D)\neq\emptyset$. 
\end{itemize}
Furthermore, there exists $i\in\{0,\ldots,w(C)\}$ such that 
$e\cap V(C)=\{\varphi_C(2i)\}$.  
Lastly, for each $j\in\{0,\ldots,w(C)\}$, $(e\setminus\{\varphi_C(2i)\})\cup\{\varphi_C(2j)\}\in E(H)$.  
\item The function $\varphi_C$ is an isomorphism from $C_3(U_{v(C)})$ onto $H[V(C)]$. 
\item If $V(C)\subsetneq V(H)$, then there exists $C'\in\mathfrak{C}(\mathscr{P}(H))\setminus\mathfrak{C}_1(\mathscr{P}(H))$ such that $C'\neq C$.
\end{enumerate}
\end{prop}

\begin{proof}
For the first assertion, suppose that $V(C)\subsetneq V(H)$. 
Since $V(C)$ is not a module of $H$, there exists $e\in E(H)$ such that 
$e\cap V(C)\neq\emptyset$ and $e\setminus V(C)\neq\emptyset$. 
It follows from Lemma~\ref{Lcritical_U2n+1_V2n+1} that 
there exists 
$i\in\{0,\ldots,w(C)\}$ such that $e\cap V(C)=\{\varphi_C(2i)\}$. 
It follows from Lemma~\ref{Lcritical} that  
$(e\setminus\{\varphi_C(2i)\})\cup\{\varphi_C(2j)\}\in E(H)$ for each 
$j\in\{0,\ldots,w(C)\}$. 
Moreover, since $|e\cap V(C)|=1$, there exists $D\in\mathfrak{C}(\mathscr{P}(H))\setminus\{C\}$ such that $e\cap V(D)\neq\emptyset$. 
We distinguish the following two cases. 
\begin{itemize}
\item Suppose that $D\in\mathfrak{C}_{{\rm even}}(\mathscr{P}(H))$. 
By Lemma~\ref{Lcritical_U2n+1_V2n+1}, $|e\cap V(D)|=2$. 
Thus $e\subseteq V(C)\cup V(D)$. 
\item Suppose that $D\in\mathfrak{C}_{{\rm odd}}(\mathscr{P}(H))$. 
By Lemma~\ref{Lcritical_U2n+1_V2n+1}, $|e\cap V(D)|=1$. 
Therefore, there exists $D'\in\mathfrak{C}(\mathscr{P}(H))\setminus\{C,D\}$ such that $e\cap V(D')\neq\emptyset$. 
It follows from Lemma~\ref{Lcritical_U2n+1_V2n+1} that 
$D'\in\mathfrak{C}_{{\rm odd}}(\mathscr{P}(H))$. 
Thus $e\subseteq V(C)\cup V(D)\cup V(D')$. 
\end{itemize}

For the second assertion, suppose for a contradiction that for every $e\in E(H)$ such that 
$e\cap V(C)\neq\emptyset$, we have $e\setminus V(C)\neq\emptyset$. 
By Lemma~\ref{Lcritical_U2n+1_V2n+1}, there exists 
$i\in\{0,\ldots,w(C)\}$ such that 
$e\cap V(C)=\{\varphi_C(2i)\}$. 
Thus, for every $e\in E(H)$, we have 
$e\cap\{\varphi_C(2i+1):i\in\{0,\ldots,w(C)-1\}\}=\emptyset$. 
Therefore, $\{\varphi_C(2i+1):i\in\{0,\ldots,w(C)-1\}\}$ is a module of $H$, and hence $H$ is decomposable. 
Consequently, there exists $e\in V(H)$ such that $e\subseteq V(C)$. 
Set $$p=\min(\{i\in\{0,\ldots,2w(C)\}:\varphi_C(i)\in e\}).$$
For a contradiction, suppose that $p$ is odd. 
It follows from Lemma~\ref{Lcritical} that 
$(e\setminus\{\varphi_C(p)\})\cup\{\varphi_C(1)\}\in E(H)$, 
which contradicts the fact that 
$V(C)\setminus\{\varphi_C(0),\varphi_C(1)\}$ is a module of $H-\varphi_C(0)$. 
Thus, there exists $i\in\{0,\ldots,w(C)\}$ such that 
$p=\varphi_C(2i)$. 
Similarly, there exists $k\in\{0,\ldots,w(C)\}$ such that 
$\max(\{i\in\{0,\ldots,2w(C)\}:\varphi_C(i)\in e\})=\varphi_C(2k)$. 
since $e\subseteq V(C)$, $i<k$. 
Consider $q\in\{0,\ldots,2w(C)\}$ such that $e=\varphi_C(2i)\varphi_C(q)\varphi_C(2k)$. 
For a contradiction, suppose that $q$ is even. 
We obtain $2i<q<2k$. 
It follows from Lemma~\ref{Lcritical} that 
$\varphi_C(2i)\varphi_C(2i+2)\varphi_C(2k)\in E(H)$, 
which contradicts the fact that 
$\{\varphi_C(2i),\varphi_C(2i+2)\}$ is a module of $H-\varphi_C(2i+1)$. 
It follows that $p$ is odd. 
Hence, there exists $j\in\{0,\ldots,w(C)-1\}$ such that 
$q=\varphi_C(2j+1)$. 
We obtain $e=\varphi_C(2i)\varphi_C(2j+1)\varphi_C(2k)$, where 
$0\leq i\leq j<k\leq w(C)$. 
It follows that 
\begin{equation}\label{E6bis_Pcritical_U2n+1}
E(H[V(C)])\subseteq\{\varphi_C(2i')\varphi_C(2j'+1)\varphi_C(2k'):0\leq i'\leq j'<k'\leq w(C)\}.
\end{equation}
It follows from Lemma~\ref{Lcritical} that 
\begin{equation}\label{E7_Pcritical_U2n+1}
\varphi_C(2i')\varphi_C(2j'+1)\varphi_C(2k')\equiv_H\varphi_C(0)\varphi_C(1)\varphi_C(2)
\end{equation} 
for $0\leq i'\leq j'<k'\leq w(C)$. 
Since $e=\varphi_C(2i)\varphi_C(2j+1)\varphi_C(2k)$, we obtain $\varphi_C(0)\varphi_C(1)\varphi_C(2)$. 
It follows from \eqref{E6bis_Pcritical_U2n+1} and \eqref{E7_Pcritical_U2n+1} that 
$$E(H[V(C)])=\{\varphi_C(2i')\varphi_C(2j'+1)\varphi_C(2k'):0\leq i'\leq j'<k'\leq w(C)\}.$$
In other words, $\varphi_C$ is an isomorphism from $C_3(U_{v(C)})$ onto $H[V(C)]$. 

For the third assertion, suppose that $V(C)\subsetneq V(H)$. 
By Theorem~\ref{thm _Critical_ST}, $U_{v(C)}$ is prime. 
Hence, $C_3(U_{v(C)})$ is prime by Theorem~\ref{thm_same_primality}. 
It follows from the second assertion above that $H[V(C)]$ is prime. 
By Theorem~\ref{prop_Eh-2}, there exist $v,w\in (V(H)\setminus V(C))\cup\underline{V(C)}$ such that $H-\{v,w\}$ is prime. 
Since $H$ is critical, we have $v\neq w$. 
Therefore, there exists 
$C'\in\mathfrak{C}(\mathscr{P}(H))\setminus\mathfrak{C}_1(\mathscr{P}(H))$ such that $v,w\in C'$. 
Lastly, suppose for a contradiction that $\underline{V(C)}\neq\emptyset$. 
There exist $c\in V(C)$ and $u\in V(H)\setminus V(C)$ such that $\{c,u\}$ is a module of $H[V(C)\cup\{u\}]$. 
Since $C$ is isomorphic to $C_3(U_{v(C)})$ by the second assertion above, 
there exists $e\in E(C)$ such that $c\in e$. 
Since $\{c,u\}$ is a module of $H[V(C)\cup\{u\}]$, we get 
$(e\setminus\{c\})\cup\{u\}\in E(H)$, which contradicts Lemma~\ref{Lcritical_U2n+1_V2n+1}. 
It follows that $\underline{V(C)}=\emptyset$. 
Thus $v,w\not\in V(C)$, so $C'\neq C$. 
\end{proof}

\begin{prop}\label{Pcritical_V2n+1}
Let $H$ be a critical and 3-hypergraph such that $v(H)\geq 5$. 
Suppose that $H$ is not circular. 
Consider $C\in\mathfrak{C}_{{\rm even}}(\mathscr{P}(H))$. 
\begin{enumerate}
\item There exists $D\in\mathfrak{C}_{{\rm odd}}(\mathscr{P}(H))$ such that 
$\varphi_C(2i)\varphi_C(2j+1)\varphi_D(2k)\in E(H)$, where $i,j\in\{0,\ldots,w(C)-1\}$, with $i\leq j$, and $k\in\{0,\ldots,w(D)\}$;
\item For each $k\in\{0,\ldots,w(D)\}$, 
the extension $\psi_C^{2k}:\{0,\ldots,v(C)\}\longrightarrow 
V(C)\cup\{\varphi_D(2k)\}$ of $\varphi_C$ defined by 
\begin{equation*}
\begin{cases}
(\psi_C^{2k})_{\restriction\{0,\ldots,v(C)-1\}}=
\varphi_{C}\\
\text{and}\\
\psi_C^{2k}(v(C))=\varphi_{D}(2k),
\end{cases}
\end{equation*}
is an isomorphism from 
$C_3(W_{v(C)+1})$ onto 
$H[V(C)\cup\{\varphi_D(2k)\}]$.
\item If $|V(H)\setminus V(C)|\geq 2$, then there exists $C'\in\mathfrak{C}(\mathscr{P}(H))\setminus\mathfrak{C}_1(\mathscr{P}(H))$ such that $C'\neq C$.
\end{enumerate}
\end{prop}

\begin{proof}
For the first assertion, since $V(C)$ is not a module of $H$, there exists $e\in E(H)$ such that 
$e\cap V(C)\neq\emptyset$ and $e\setminus V(C)\neq\emptyset$. 
By Lemma~\ref{Lcritical_U2n+1_V2n+1}, there exist $i,j\in\{0,\ldots,w(C)-1\}$ such that 
$i\leq j$ and $e\cap V(C)=\{\varphi_C(2i),\varphi_C(2j+1)\}$. 
Since $|e\cap V(C)|=2$, there exists $D\in\mathfrak{C}(\mathscr{P}(H))\setminus\{C\}$ such that $e\cap V(D)\neq\emptyset$. 
It follows from Lemma~\ref{Lcritical_U2n+1_V2n+1} that 
$D\in\mathfrak{C}_{{\rm odd}}(\mathscr{P}(H))$ and $e\cap V(D)=\{\varphi_D(2k)\}$, where $k\in\{0,\ldots,w(D)\}$. 

For the second assertion, it follows from the first assertion above that 
there exists $D\in\mathfrak{C}_{{\rm odd}}(\mathscr{P}(H))$ such that 
\begin{equation}\label{E0_Pcritical_V2n+1}
\varphi_C(2i_0)\varphi_C(2j_0+1)\varphi_D(2k_0)\in E(H), 
\end{equation}
where $i_0,j_0\in\{0,\ldots,w(C)-1\}$, with $i_0\leq j_0$, and $k_0\in\{0,\ldots,w(D)\}$. 
It follows from Lemma~\ref{Lcritical} that 
\begin{equation}\label{E2_Pcritical_V2n+1}
\varphi_C(2i_0)\varphi_C(2j_0+1)\varphi_D(2k)\in E(H)
\end{equation}
for each $k\in\{0,\ldots,w(D)\}$. 
Let $k\in\{0,\ldots,w(D)\}$. 
We prove that 
\begin{align}\label{E3_Pcritical_V2n+1}
E(H[V(C)\cup\{\varphi_D(2k)\}])=\{&\varphi_C(2i)\varphi_C(2j+1)\varphi_D(2k):\\
&0\leq i\leq j\leq w(C)-1\}.\nonumber
\end{align}
Consider $i,j\in\{0,\ldots,w(C)-1\}$ such that $i\leq j$. 
It follows from Lemma~\ref{Lcritical} that 
\begin{equation}\label{E4_Pcritical_V2n+1}
\varphi_C(2i)\varphi_C(2j+1)\varphi_D(2k)\equiv_H
\varphi_C(0)\varphi_C(1)\varphi_D(2k).
\end{equation}

By \eqref{E2_Pcritical_V2n+1} and \eqref{E4_Pcritical_V2n+1}, we have 
$\varphi_C(0)\varphi_C(1)\varphi_D(2k)\in E(H)$. 
It follows from \eqref{E4_Pcritical_V2n+1} that 
$\varphi_C(2i)\varphi_C(2j+1)\varphi_D(2k)\in E(H)$. 
Therefore, 
\begin{align}\label{E5_Pcritical_V2n+1}
E(H[V(C)\cup\{\varphi_D(2k)\}])\supseteq\{&\varphi_C(2i)\varphi_C(2j+1)\varphi_D(2k):\\
&0\leq i\leq j\leq w(C)-1\}.\nonumber
\end{align}
Conversely, consider $e\in E(H[V(C)\cup\{\varphi_D(2k)\}])$. 
For a contradiction, suppose that $\varphi_D(2k)\not\in e$. 
There exist $p,q,r\in\{0,\ldots,2w(C)-1\}$, with $p<q<r$, such that 
$e=\varphi_C(p)\varphi_C(q)\varphi_D(r)$. 
If $p$ is odd, then $\varphi_C(1)\varphi_C(q)\varphi_D(r)\in E(H)$, which contradicts the fact that $V(H)\setminus\{\varphi_C(0),\varphi_C(1)\}$ is a module of 
$H-\varphi_C(0)$. 
Hence, $p=2i$, where $i\in\{0,\ldots,w(C)-1\}$. 
Similarly, we have $r=2l+1$, where $l\in\{i,\ldots,w(C)-1\}$. 
If $q$ is even, then $\varphi_C(2i)\varphi_C(2i+2)\varphi_D(2l+1)\in E(H)$, which contradicts the fact that $\{\varphi_C(2i),\varphi_C(2i+2)\}$ is a module of 
$H-\varphi_C(2i+1)$. 
If $q$ is odd, then $\varphi_C(2i)\varphi_C(2l-1)\varphi_D(2l+1)\in E(H)$, which contradicts the fact that $\{\varphi_C(2l-1),\varphi_C(2l+1)\}$ is a module of 
$H-\varphi_C(2l)$. 
Consequently, we obtain  
$\varphi_D(2k)\in e$. 
It follows from Lemma~\ref{Lcritical_U2n+1_V2n+1} that there exist $i,j\in\{0,\ldots,w(C)-1\}$ such that 
$i\leq j$ and $e\cap V(C)=\{\varphi_C(2i),\varphi_C(2j+1)\}$. 
Thus,   
\begin{align}\label{E6_Pcritical_V2n+1}
E(H[V(C)\cup\{\varphi_D(2k)\}])\subseteq\{&\varphi_C(2i)\varphi_C(2j+1)\varphi_D(2k):\\
&0\leq i\leq j\leq w(C)-1\}.\nonumber
\end{align}
It follows from \eqref{E5_Pcritical_V2n+1} and \eqref{E6_Pcritical_V2n+1} that \eqref{E3_Pcritical_V2n+1} holds. 
Therefore, the extension $\psi_C^{2k}$ of $\varphi_C$ is an isomorphism from 
$C_3(W_{v(C)+1})$ onto $H[V(C)\cup\{\varphi_D(2k)\}]$. 

For the third assertion, suppose that $|V(H)\setminus V(C)|\geq 2$. 
Let $k\in\{0,\ldots,w(D)\}$. 
Set $Y=V(C)\cup\{\varphi_D(2k)\}$. 
We have $Y\subsetneq V(H)$. 
By Theorem~\ref{thm _Critical_ST}, $W_{v(C)+1}$ is prime. 
Hence, $C_3(W_{v(C)+1})$ is prime by Theorem~\ref{thm_same_primality}. 
It follows from the second assertion above that $H[Y]$ is prime. 
By Proposition~\ref{prop_Eh-2}, there exist $v,w\in (V(H)\setminus Y)\cup\underline{Y}$ such that $H-\{v,w\}$ is prime. 
Since $H$ is critical, we have $v\neq w$. 
Therefore, there exists 
$C'\in\mathfrak{C}(\mathscr{P}(H))\setminus\mathfrak{C}_1(\mathscr{P}(H))$ such that $v,w\in C'$. 
Lastly, suppose for a contradiction that $V(C)\cap\underline{Y}\neq\emptyset$. 
There exist $c\in V(C)$ and $u\in V(H)\setminus Y$ such that $\{c,u\}$ is a module of 
$H[Y\cup\{u\}]$. 
By the second assertion above, 
there exists $d\in V(C)\setminus\{c\}$ such that $cd\varphi_D(2k)\in E(H[Y])$. 
Since $\{c,u\}$ is a module of $H[V(C)\cup\{u\}]$, we get 
$ud\varphi_D(2k)\in E(H[Y])$, which contradicts Lemma~\ref{Lcritical_U2n+1_V2n+1}. 
It follows that $\underline{Y}\subseteq V(H)\setminus V(C)$. 
Thus $v,w\not\in V(C)$, so $C'\neq C$. 
\end{proof}

\begin{proof}[Proof of Theorem~\ref{thm_1_main}]
Let $H$ be a non circular and critical and 3-hypergraph such that $v(H)\geq 5$. 
By Corollary~\ref{C_critical_T2n+1}, all the components of $\mathscr{P}(H)$ are paths. 
We associate with $H$ the hypergraph $\underline{H}$ defined on 
$\mathfrak{C}(\mathscr{P}(H))$ as  follows 
\begin{enumerate}
\item given distinct $C,D\in\mathfrak{C}(\mathscr{P}(H))$, $CD\in E(\underline{H})$ if 
$C\in\mathfrak{C}_{{\rm even}}(\mathscr{P}(H))$, 
$D\in\mathfrak{C}_{{\rm odd}}(\mathscr{P}(H))$, and 
there exists $e\in E(H)$ such that $|e\cap V(C)|=2$ and $|e\cap V(D)|=1$;
\item given distinct $I,J,K\in\mathfrak{C}(\mathscr{P}(H))$, $IJK\in E(\underline{H})$ if 
$I,J,K\in\mathfrak{C}_{{\rm odd}}(\mathscr{P}(H))$ and 
$|e\cap V(I)|=|e\cap V(J)|=|e\cap V(K)|=1$. 
\end{enumerate}
It follows from Propositions~\ref{Pcritical_U2n+1} and \ref{Pcritical_V2n+1} that 
$H=\mathscr{P}(H)\bullet\underline{H}$. 
\end{proof}

\section{Proof of Proposition~\ref{Prop1_Critical_hyper}}

We use the following preliminary result (see \cite[Fact~20]{BCIZ}

\begin{lem}\label{fact1_same_block}
Let $H$ be a 3-hypergraph. Consider $X\subsetneq V(H)$ such that $H[X]$ is prime. 
Let $M$ be a module of $H$. We have $M\cap X=\emptyset$, $M\supseteq X$ or $M\cap X=\{y\}$, where $y\in X$. Moreover, the following assertions hold. 
\begin{enumerate}
\item If $M\cap X=\emptyset$, then all the elements of $M$ belong to the same block of $p_{(H,X)}$. 
\item If $M\supseteq X$, then all the elements of $V(H)\setminus M$ belong to $\langle X\rangle_H$. 
\item If $M\cap X=\{y\}$, where $y\in X$, then all the elements of 
 $M\setminus\{y\}$ belong to $X_H(y)$. 
\end{enumerate}
\end{lem}

In this section and the next one, we consider a graph $\Gamma$ satisfying \eqref{critical_construction_0} and \eqref{critical_construction_1}. 
We consider also a $\{2,3\}$-hypergraph $\mathbb{H}$ defined on $\mathfrak{C}(\Gamma)$ satisfying 
\eqref{critical_construction_2}. 
We use the following notation. 

\begin{nota}
For $\mathbb{W}\subseteq V(\mathbb{H})$, set 
$\overline{\mathbb{W}}=\bigcup_{C\in\mathbb{W}}V(C)$. 
Conversely, for $W\subseteq V(H_{(\Gamma,\mathbb{H})})$, set 
$W/\mathfrak{C}(\Gamma)=\{C\in\mathfrak{C}(\Gamma):V(C)\cap W\neq\emptyset\}$. 
\end{nota}

Both next remarks are useful. 

\begin{rem}\label{v(C)_odd}
Let $C\in\mathfrak{C}_{{\rm odd}}(\Gamma)\setminus\mathfrak{C}_1(\Gamma)$. 
First, suppose that $w(C)=1$. 
It follows from the definition of $\Gamma\bullet\mathbb{H}$ that 
$E((\Gamma\bullet\mathbb{H})[V(C)])=\{V(C)\}$. 
Therefore, 
$(\Gamma\bullet\mathbb{H})[V(C)]$ is prime. 
Since $v(C)=3$, the next three assertions are obvious
\begin{enumerate}
\item $V(C)\setminus\{\varphi_C(0),\varphi_C(1)\}$ 
is a module of $(\Gamma\bullet\mathbb{H})[V(C)]-\varphi_C(0)$;
\item $V(C)\setminus\{\varphi_C(1),\varphi_C(2)\}$ 
is a module of $(\Gamma\bullet\mathbb{H})[V(C)]-\varphi_C(2)$;
\item $\{\varphi_C(0),\varphi_C(2)\}$ is a module of $(\Gamma\bullet\mathbb{H})[V(C)]-\varphi_C(1)$.
\end{enumerate}

Second, suppose that $w(C)\geq 2$. 
It follows from the definition of $\Gamma\bullet\mathbb{H}$ that $\varphi_C$ is an isomorphism from $C_3(U_{v(C)})$ onto $(\Gamma\bullet\mathbb{H})[V(C)]$. 
By Theorem~\ref{thm _Critical_ST}, $U_{v(C)}$ is critical. 
By Theorem~\ref{thm_same_primality}, $C_3(U_{v(C)})$ is critical and 
$\mathscr{P}(C_3(U_{v(C)}))=\mathscr{P}(U_{v(C)})$. 
Furthermore, 
we have $\mathscr{P}(U_{v(C)})=P_{v(C)}$ by Proposition~\ref{prop_U2n+1}. 
Since $\varphi_C$ is an isomorphism from $C_3(U_{v(C)})$ onto $(\Gamma\bullet\mathbb{H})[V(C)]$, 
$(\Gamma\bullet\mathbb{H})[V(H)]$ is critical and 
$$\mathscr{P}((\Gamma\bullet\mathbb{H})[V(C)])=(V(C),\{\varphi_C(i)\varphi_C(i+1):0\leq i<2w(C)\}).$$
\end{rem}

\begin{rem}\label{v(C)_even}
Let $C\in\mathfrak{C}_{{\rm even}}(\Gamma)$ such that $w(C)\geq 1$. 
Suppose that there exists $e\in E(\mathbb{H})$ such that $C\in e$. 
We have $e=CD$, where $D\in\mathfrak{C}_{{\rm odd}}(\Gamma)$. 
Let $k\in\{0,\ldots,w(D)\}$.  
First, suppose that $w(C)=1$. 
It follows from the definition of $\Gamma\bullet\mathbb{H}$ that 
$E((\Gamma\bullet\mathbb{H})[V(C)\cup\{\varphi_D(2k)\}])=\{V(C)\cup\{\varphi_D(2k)\}\}$. 
Therefore, 
$(\Gamma\bullet\mathbb{H})[V(C)\cup\{\varphi_D(2k)\}]$ is prime. 

Second, suppose that $w(C)\geq 2$. 
Let $\psi_C^{2k}:\{0,\ldots,v(C)\}\longrightarrow 
V(C)\cup\{\varphi_D(2k)\}$ satisfying $(\psi_C^{2k})_{\restriction\{0,\ldots,v(C)-1\}}=\varphi_{C}$ and 
$\psi_C^{2k}(v(C))=\varphi_{D}(2k)$. 
By \eqref{E1_Rbis_Critical_hyper}, $\psi_C^{2k}$ is an isomorphism from 
$C_3(W_{v(C)+1})$ onto 
$(\Gamma\bullet\mathbb{H})[V(C)\cup\{\varphi_D(2k)\}]$. 
By Theorem~\ref{thm _Critical_ST}, $W_{v(C)+1}$ is critical. 
By Theorem~\ref{thm_same_primality}, $C_3(W_{v(C)+1})$ is critical and 
$\mathscr{P}(C_3(W_{v(C)+1}))=\mathscr{P}(W_{v(C)+1})$. 
Furthermore, it follows from Proposition~\ref{prop_W2n+1} that $\mathscr{P}(W_{v(C)+1})-2n=P_{v(C)}$ and 
$N_{\mathscr{P}(W_{v(C)+1})}(2n)=\emptyset$. 
Since $\psi_C^{2k}$ is an isomorphism from 
$C_3(W_{v(C)+1})$ onto 
$(\Gamma\bullet\mathbb{H})[V(C)\cup\{\varphi_D(2k)\}]$, 
$(\Gamma\bullet\mathbb{H})[V(C)\cup\{\varphi_D(2k)\}]$ is critical and we have 
\begin{align*}
\mathscr{P}((\Gamma\bullet\mathbb{H})[V(C)\cup\{\varphi_D(2k)\}])&-\varphi_D(2k)=\\
&(V(C),\{\varphi_C(i)\varphi_C(i+1):0\leq i<2w(C)-1\})
\end{align*}
and 
$N_{\mathscr{P}((\Gamma\bullet\mathbb{H})[V(C)\cup\{\varphi_D(2k)\}])}(\varphi_D(2k))=\emptyset$. 
\end{rem}

If $\mathbb{H}$ is disconnected, then $\Gamma\bullet\mathbb{H}$ is decomposable, whence the necessity of  Assertion (C1) in Proposition~\ref{Prop1_Critical_hyper}. 
Indeed, we have

\begin{lem}\label{lem1bis_Crticial_hyper}
If $\mathbb{C}$ is a component of $\mathbb{H}$, then 
$\overline{V(\mathbb{C})}$ and 
$V(\Gamma\bullet\mathbb{H})\setminus\overline{V(\mathbb{C})}$ are modules of 
$\Gamma\bullet\mathbb{H}$. 
\end{lem}

\begin{proof}
For a contradiction, suppose that there exists $e\in E(\Gamma\bullet\mathbb{H})$ such that 
$e\cap\overline{V(\mathbb{C})}\neq\emptyset$ 
and 
$e\cap(V(\Gamma\bullet\mathbb{H})\setminus\overline{V(\mathbb{C})})\neq\emptyset$. 
There exist $C\in V(\mathbb{C})$ and $D\not\in V(\mathbb{C})$ such that $e\cap V(C)\neq\emptyset$ and $e\cap V(D)\neq\emptyset$. 
It follows from the definition of $\Gamma\bullet\mathbb{H}$ that there exists 
$\varepsilon\in E(\mathbb{H})$ such that $C,D\in\varepsilon$, which contradicts the fact that 
$\mathbb{C}$ is a component of $\mathbb{H}$. 
Consequently, for each $e\in E(\Gamma\bullet\mathbb{H})$, we have $e\subseteq\overline{V(\mathbb{C})}$ or 
$e\subseteq(V(\Gamma\bullet\mathbb{H})\setminus\overline{V(\mathbb{C})})\neq\emptyset$. 
It follows that $\overline{V(\mathbb{C})}$ and 
$V(\Gamma\bullet\mathbb{H})\setminus\overline{V(\mathbb{C})}$ are modules of 
$\Gamma\bullet\mathbb{H}$. 
\end{proof}

Each edge of $\mathbb{H}$ induces a prime subhypergraph of $\Gamma\bullet\mathbb{H}$. 
Precisely, we have 

\begin{lem}\label{Cl0_Rbis_Critical_hyper}
For every $\varepsilon\in E(\mathbb{H})$, 
$(\Gamma\bullet\mathbb{H})[\overline{\varepsilon}]$ is prime. 
\end{lem}

\begin{proof}
Let $\varepsilon\in E(\mathbb{H})$. 
First, suppose that $|\varepsilon|=2$. 
There exist $C\in\mathfrak{C}_{{\rm even}}(\Gamma)$ and 
$D\in\mathfrak{C}_{{\rm odd}}(\Gamma)$ such that 
$\varepsilon=CD$. 
Hence, we have to show that 
$(\Gamma\bullet\mathbb{H})[V(C)\cup V(D)]$ is prime. 
If $w(D)=0$, then 
$(\Gamma\bullet\mathbb{H})[V(C)\cup V(D)]$ is prime by Remark~\ref{v(C)_even}. 
Suppose that $w(D)\geq 1$. 
Set $X=V(D)$. 
By Remark~\ref{v(C)_odd}, $(\Gamma\bullet\mathbb{H})[X]$ is prime. 
We have $V(C)\subseteq\langle X\rangle_{\Gamma\bullet\mathbb{H}}$ 
(see Notation~\ref{nota_p_hyper}). 
Let $M$ be a module of 
$(\Gamma\bullet\mathbb{H})[X\cup V(C)]$ such that $|M|\geq 2$. 
We have to show that $M=X\cup V(C)$. 
By Lemma~\ref{fact1_same_block}, we have 
$M\cap X=\emptyset$, $M\supseteq X$, or $M\cap X=\{y\}$, where $y\in X$. 
It follows from the third assertion of Lemma~\ref{fact1_same_block} that $|M\cap X|\neq 1$. 
Now, suppose that $M\cap X=\emptyset$. 
Hence, we have $M\subseteq V(C)$. 
By Remark~\ref{v(C)_even}, $(\Gamma\bullet\mathbb{H})[V(C)\cup\{\varphi_{D}(0)\}]$ is prime. 
Since $|M|\geq 2$ and $M\subseteq V(C)$, we obtain 
$M=V(C)\cup\{\varphi_{D}(0)\}$, which contradicts 
$M\cap X=\emptyset$. 
Therefore, we have $M\supseteq X$. 
Let $i,j\in\{0,\ldots,w(C)-1\}$ with $i\leq j$. 
We have $\varphi_{C}(2i)\varphi_{C}(2j+1)\varphi_{D}(0)\in E(H)$. 
Since $M$ is a module of 
$(\Gamma\bullet\mathbb{H})[X\cup V(C)]$ and 
$\varphi_{C}(2i)\varphi_{C}(2j+1)\varphi_{D}(1)\not\in E(H)$, we obtain 
$\varphi_{C}(2i),\varphi_{C}(2j+1)\in M$. 
It follows that $M=X\cup V(C)$. 

Second, suppose that $|e|=3$. 
There exist $I,J,K\in\mathfrak{C}_{{\rm odd}}(\Gamma)$ such that 
$e=IJK$. 
Hence, we have to show that 
$(\Gamma\bullet\mathbb{H})[V(I)\cup V(J)\cup V(K)]$ is prime. 
If $w(I)=w(J)=w(K)=0$, then $|V(I)\cup V(J)\cup V(K)|=3$, and hence 
$E((\Gamma\bullet\mathbb{H})[V(I)\cup V(J)\cup V(K)])=
\{\varphi_{I}(0)\varphi_{J}(0)\varphi_{K}(0)\}$. 
It follows that 
$(\Gamma\bullet\mathbb{H})[V(I)\cup V(J)\cup V(K)]$ is prime. 
Now, suppose that $w(I)>0$, $w(J)>0$, or $w(K)>0$. 
For instance, assume that $w(I)>0$. 
By Remark~\ref{v(C)_odd}, $(\Gamma\bullet\mathbb{H})[V(I)]$ is prime. 
Set $X=V(I)$. 
We have $V(I)\cup V(J)\cup V(K)\subseteq\langle X\rangle_{\Gamma\bullet\mathbb{H}}$. 
Let $M$ be a module of $(\Gamma\bullet\mathbb{H})[X\cup V(J)\cup V(K)]$ such that 
$|M|\geq 2$. 
We have to show that 
$M=X\cup V(J)\cup V(K)$. 
It follows from Lemma~\ref{fact1_same_block} that $M\cap X=\emptyset$ or $X\subseteq M$. 

For a contradiction, suppose that $M\cap X=\emptyset$. 
For $0\leq j\leq w(J)$ and $0\leq k\leq w(K)$, we have 
$\varphi_{I}(0)\varphi_{J}(2j)\varphi_{K}(2k)\in E(\Gamma\bullet\mathbb{H})$. 
It follows that 
$M\cap\{\varphi_{J}(2j):0\leq j\leq w(J)\}=\emptyset$ 
or 
$M\cap\{\varphi_{K}(2k):0\leq k\leq w(K)\}=\emptyset$.
For instance, assume that 
\begin{equation}\label{E2_Cl0_Rbis_Critical_hyper}
M\cap\{\varphi_{J}(2j):0\leq j\leq w(J)\}=\emptyset.
\end{equation}
We distinguish the following two cases. 
\begin{enumerate}
\item Suppose that 
$M\cap\{\varphi_{J}(2j+1):0\leq j\leq w(J)-1\}
\neq\emptyset$. 
Hence, there exists $j\in\{0,\ldots,w(J)-1\}$ such that 
$\varphi_{J}(2j+1)\in M$. 
We have 
$\varphi_{J}(2j)\varphi_{J}(2j+1)\varphi_{J}(2j+2)\in E(H)$. 
By \eqref{E2_Cl0_Rbis_Critical_hyper}, $\varphi_{J}(2j),
\varphi_{J}(2j+2)\not\in M$. 
Since $M$ is a module of $(\Gamma\bullet\mathbb{H})[X\cup V(J)\cup V(K)]$, we obtain 
$\varphi_{J}(2j)x\varphi_{J}(2j+2)\in E(H)$ for every $x\in M$. 
For each $v\in V(K)$, we have 
$\varphi_{J}(2j)v\varphi_{J}(2j+2)\not\in E(H)$. 
It follows that $M\cap V(K)=\emptyset$. 
Therefore, $M\subseteq V(J)$, and hence 
$M$ is a module of $(\Gamma\bullet\mathbb{H})[X\cup V(J)]$. 
Since $|M|\geq 2$, $w(J)\geq 1$.  
By Remark~\ref{v(C)_odd}, 
$(\Gamma\bullet\mathbb{H})[V(J)]$ is prime. 
Therefore, $M=V(J)$, which contradicts \eqref{E2_Cl0_Rbis_Critical_hyper}. 
\item Suppose that 
$M\cap\{\varphi_{J}(2j+1):0\leq j\leq w(J)-1\}=\emptyset$. 
By \eqref{E2_Cl0_Rbis_Critical_hyper}, $M\cap V(J)=\emptyset$, and hence $M\subseteq V(K)$. 
As previously for $J$, we obtain $w(k)\geq 1$ and $M=V(K)$, which is impossible because 
$\varphi_{I}(0)\varphi_{J}(0)\varphi_{K}(0)\in E(H)$ and 
$\varphi_{I}(0)\varphi_{J}(0)\varphi_{K}(1)\not\in E(H)$. 
\end{enumerate}

It follows that $X\subseteq M$. 
Recall that $\varphi_{I}(0)\varphi_{J}(2j)\varphi_{K}(2k)\in E(\Gamma\bullet\mathbb{H})$ for 
$0\leq j\leq w(J)$ and $0\leq k\leq w(K)$. 
Let $j\in\{0,\ldots,w(J)\}$ and 
$k\in\{0,\ldots,w(J)\}$. 
We have $\varphi_{I}(0)\in M$ because $X\subseteq M$. 
Since 
$\varphi_{I}(1)\varphi_{J}(2j)\varphi_{K}(2k)\not\in E(\Gamma\bullet\mathbb{H})$ and $M$ is a module of 
$(\Gamma\bullet\mathbb{H})[X\cup V(J)\cup V(K)]$, we obtain 
$\varphi_{J}(2j),\varphi_{K}(2k)\in M$.  
It follows that $X\cup\{\varphi_{J}(2j):0\leq j\leq w(J)\}\cup\{\varphi_{K}(2k):0\leq k\leq w(K)\}\subseteq M$. 
Clearly, $V(J)\subseteq M$ when 
$w(J)=0$. 
Suppose that $w(J)\geq 1$. 
By Remark~\ref{v(C)_odd}, 
$(\Gamma\bullet\mathbb{H})[V(J)]$ is prime. 
Since $M\cap V(J)$ is a module of $(\Gamma\bullet\mathbb{H})[V(J)]$ such that 
$\{\varphi_{J}(2j):0\leq j\leq w(J)\}\subseteq 
M\cap V(J)$, we obtain $V(J)\subseteq M$. 
Similarly, $V(K)\subseteq M$. 
Consequently, $M=X\cup V(J)\cup V(K)$. 
\end{proof}

In the next four results, we compare the module of $\mathbb{H}$ with those of $\Gamma\bullet\mathbb{H}$.

\begin{lem}\label{lem1_Crticial_hyper}
Suppose that $\mathbb{H}$ is connected. 
For each nontrivial module $M$ of $\Gamma\bullet\mathbb{H}$, we have 
$M=\overline{M/\mathfrak{C}(\Gamma)}$, 
$|M/\mathfrak{C}(\Gamma)|\geq 2$, and 
$M/\mathfrak{C}(\Gamma)$ is a module of $\mathbb{H}$. 
\end{lem}

\begin{proof}
First, we prove that $V(C)\subseteq M$ for every $C\in M/\mathfrak{C}(\Gamma)$. 
To begin, suppose that $C\in\mathfrak{C}_{{\rm odd}}(\Gamma)$. 
If $w(C)=0$, then we clearly have $V(C)\subseteq M$. 
Hence, suppose that $w(C)\geq 1$. 
By Remark~\ref{v(C)_odd}, $(\Gamma\bullet\mathbb{H})[V(C)]$ is prime. 
For a contradiction, suppose that $|V(C)\cap M|=1$. 
Denote by $c$ the unique element of $V(C)\cap M$. 
Since $(\Gamma\bullet\mathbb{H})[V(C)]$ is prime, there exist distinct $d,d'\in V(C)\setminus M$ such that 
$cdd'\in E(\Gamma\bullet\mathbb{H})$. 
Since $|M|\geq 2$, there exists $v\in M\setminus V(C)$. 
Since $M$ is a module of $\Gamma\bullet\mathbb{H}$, we obtain 
$vdd'\in E(H_{(\Gamma,\mathbb{H})})$, which contradicts the definition of 
$\Gamma\bullet\mathbb{H}$. 
It follows that $|V(C)\cap M|\geq 2$. 
Since $(\Gamma\bullet\mathbb{H})[V(C)]$ is prime, $V(C)\subseteq M$. 
Now, suppose that $C\in\mathfrak{C}_{{\rm even}}(\Gamma)$. 
Since $\mathbb{H}$ is connected, there exists 
$D\in\mathfrak{C}_{{\rm odd}}(\Gamma)$ such that $CD\in E(\mathbb{H})$. 
For a contradiction, suppose that $|V(C)\cap M|=1$. 
Denote by $c$ the unique element of $V(C)\cap M$. 
There exists $c'\in V(C)\setminus M$ such that 
$cc'\varphi_D(0)\in E(\Gamma\bullet\mathbb{H})$. 
By Remark~\ref{v(C)_even}, 
$(\Gamma\bullet\mathbb{H})[V(C)\cup\{\varphi_D(0)\}]$ is prime. 
Hence, if $|(V(C)\cup\{\varphi_D(0)\})\cap M|\geq 2$, then 
$V(C)\cup\{\varphi_D(0)\}\subseteq M$, which contradicts $|V(C)\cap M|=1$. 
Therefore, $|(V(C)\cup\{\varphi_D(0)\})\cap M|=1$. 
Since $|M|\geq 2$, there exists $v\in M\setminus(V(C)\cup\{\varphi_D(0)\})$. 
Since $M$ is a module of $\Gamma\bullet\mathbb{H}$, we obtain 
$vc'\varphi_D(0)\in E(\Gamma\bullet\mathbb{H})$, which contradicts the definition of 
$\Gamma\bullet\mathbb{H}$. 
It follows that $|V(C)\cap M|\geq 2$. 
Since $(\Gamma\bullet\mathbb{H})[V(C)\cup\{\varphi_D(0)\}]$ is prime, we obtain 
$V(C)\cup\{\varphi_D(0)\}\subseteq M$. 

Second, we prove that $|M/\mathfrak{C}(\Gamma)|\geq 2$. 
If $v(C)=1$ for each $C\in M/\mathfrak{C}(\Gamma)$, then $|M/\mathfrak{C}(\Gamma)|\geq 2$ because $|M|\geq 2$. 
Hence, consider $C\in M/\mathfrak{C}(\Gamma)$ such that $v(C)\geq 2$. 
Since $\mathbb{H}$ is connected, there exists $\varepsilon\in E(\mathbb{H})$ such that $C\in\varepsilon$. 
By Lemma~\ref{Cl0_Rbis_Critical_hyper}, 
$(\Gamma\bullet\mathbb{H})[\overline{\varepsilon}]$ is prime. 
Since $V(C)\subseteq M$ and $v(C)\geq 2$, we get $\overline{\varepsilon}\subseteq M$. 
Thus $e\subseteq (M/\mathfrak{C}(\Gamma))$, so 
$|M/\mathfrak{C}(\Gamma)|\geq 2$. 

Third, we prove that $M/\mathfrak{C}(\Gamma)$ is a module of $\mathbb{H}$. 
Let $\varepsilon\in E(\mathbb{H})$ such that $\varepsilon\cap(M/\mathfrak{C}(\Gamma))\neq\emptyset$ and 
$\varepsilon\setminus(M/\mathfrak{C}(\Gamma))\neq\emptyset$. 
We distinguish the following two cases. 
\begin{itemize}
\item Suppose that $|\varepsilon|=2$. 
There exist $C\in\mathfrak{C}_{{\rm even}}(\Gamma)$ and 
$D\in\mathfrak{C}_{{\rm odd}}(\Gamma)$ such that $\varepsilon=CD$.  
Thus, there exist distinct $c,c'\in V(C)$ and $d\in V(D)$ such that 
$cc'd\in E(\Gamma\bullet\mathbb{H})$. 
For a contradiction, suppose that $C\in M/\mathfrak{C}(\Gamma)$. 
Since $\varepsilon\setminus(M/\mathfrak{C}(\Gamma))\neq\emptyset$, we have 
$D\not\in M/\mathfrak{C}(\Gamma)$. 
By the first point above, $V(C)\subseteq M$ and $V(D)\cap M=\emptyset$. 
We obtain $cc'd\in E(\Gamma\bullet\mathbb{H})$, with $c,c'\in M$ and $d\not\in M$, which contradicts the fact that $M$ is a module of $\Gamma\bullet\mathbb{H}$. 
It follows that $D\in M/\mathfrak{C}(\Gamma)$ and $C\not\in M/\mathfrak{C}(\Gamma)$. 
By the first point above, $V(D)\subseteq M$ and $V(C)\cap M=\emptyset$. 
We obtain $cc'd\in E(\Gamma\bullet\mathbb{H})$, with $c,c'\not\in M$ and $d\in M$. 
Since $M$ is a module of $\Gamma\bullet\mathbb{H}$, we obtain 
\begin{equation}\label{E1_lem1_Crticial_hyper}
\text{$cc'x\in E(H_{(\Gamma,\mathbb{H})})$ for every $x\in M$.}
\end{equation}
Consider $D'\in M/\mathfrak{C}(\Gamma)$. 
By the first point above, $V(D')\subseteq M$. 
Let $d'\in V(D')$. 
By \eqref{E1_lem1_Crticial_hyper}, 
$cc'd'\in E(\Gamma\bullet\mathbb{H})$. 
It follows from the definition of $\Gamma\bullet\mathbb{H}$ that 
$CD'\in E(\mathbb{H})$. 
\item Suppose that $|\varepsilon|=3$. 
There exist distinct $I,J,K\in\mathfrak{C}_{{\rm odd}}(\Gamma)$ such that 
$\varepsilon=IJK$. 
Moreover, we can assume that 
$I\in M/\mathfrak{C}(\Gamma)$ and $K\not\in M/\mathfrak{C}(\Gamma)$. 
By the first point above, $V(I)\subseteq M$ and $V(K)\cap M=\emptyset$. 
By definition of $\Gamma\bullet\mathbb{H}$, we obtain 
$\varphi_I(0)\varphi_{J}(0)\varphi_{K}(0)\in E(\Gamma\bullet\mathbb{H})$. 
Since $M$ is a module of $\Gamma\bullet\mathbb{H}$, we have $\varphi_{J}(0)\not\in M$. 
It follows from the first point above that $V(J)\cap M=\emptyset$. 
Thus, $\varphi_I(0)\varphi_{J}(0)\varphi_{K}(0)\in E(\Gamma\bullet\mathbb{H})$, with 
$\varphi_I(0)\in M$ and $\varphi_{J}(0),\varphi_{K}(0)\not\in M$. 
Since $M$ is a module of $\Gamma\bullet\mathbb{H}$, we obtain 
\begin{equation}\label{E2_lem1_Crticial_hyper}
\text{$x\varphi_{J}(0)\varphi_{K}(0)\in E(\Gamma\bullet\mathbb{H})$ for every 
$x\in M$.}
\end{equation}
Consider any $L\in M/\mathfrak{C}(\Gamma)$. 
By the first point above, $V(L)\subseteq M$
Let $d\in V(L)$. 
By \eqref{E2_lem1_Crticial_hyper}, 
$d\varphi_{J}(0)\varphi_{K}(0)\in E(\Gamma\bullet\mathbb{H})$. 
Since $J,K,L$ are distinct element of $\mathfrak{C}(\Gamma)$, it 
follows from the definition of $\Gamma\bullet\mathbb{H}$ that 
$JKL\in E(\mathbb{H})$. 
\end{itemize}
Consequently, $M/\mathfrak{C}(\Gamma)$ is a module of $\mathbb{H}$. 
\end{proof}

\begin{lem}\label{lem2_Crticial_hyper}
Suppose that $\mathbb{H}$ is connected. 
For a nontrivial module $\mathbb{M}$ of $\mathbb{H}$, we have 
$\mathbb{M}\subseteq\mathfrak{C}_{{\rm even}}(\Gamma)$ or 
$\mathbb{M}\subseteq\mathfrak{C}_{{\rm odd}}(\Gamma)$. 
\end{lem}

\begin{proof}
Suppose that $\mathbb{M}\setminus\mathfrak{C}_{{\rm odd}}(\Gamma)\neq\emptyset$. 
We have to show that $\mathbb{M}\subseteq\mathfrak{C}_{{\rm even}}(\Gamma)$. 
Since $\mathbb{M}\setminus\mathfrak{C}_{{\rm odd}}(\Gamma)\neq\emptyset$, there exists $C\in\mathbb{M}\cap\mathfrak{C}_{{\rm even}}(\Gamma)$. 
Since $\mathbb{H}$ is connected and $\mathbb{M}\subsetneq V(\mathbb{H})$, there exists $\varepsilon\in E(\mathbb{H})$ such that $\varepsilon\cap\mathbb{M}\neq\emptyset$ and $\varepsilon\setminus\mathbb{M}\neq\emptyset$. 
Furthermore, since $\mathbb{M}$ is a module of $\mathbb{H}$, we have $|\varepsilon\cap\mathbb{M}|=1$. 
It follows that $(\varepsilon\setminus\mathbb{M})\cup\{C\}\in E(\mathbb{H})$. 
Since $C\in\mathbb{M}\cap\mathfrak{C}_{{\rm even}}(\Gamma)$, we obtain 
$|\varepsilon\setminus\mathbb{M}|=1$ and 
$\varepsilon\setminus\mathbb{M}\subseteq\mathfrak{C}_{{\rm odd}}(\Gamma)$. 
Lastly, consider $D\in\mathbb{M}$. 
Since $\mathbb{M}$ is a module of $\mathbb{H}$, we have 
$(\varepsilon\setminus\mathbb{M})\cup\{D\}\in E(\mathbb{H})$. 
Since $|(\varepsilon\setminus\mathbb{M})\cup\{D\}|=2$ and 
$(\varepsilon\setminus\mathbb{M})\subseteq\mathfrak{C}_{{\rm odd}}(\Gamma)$, we obtain 
$D\in\mathbb{M}\cap\mathfrak{C}_{{\rm even}}(\Gamma)$. 
Consequently, 
$\mathbb{M}\subseteq\mathfrak{C}_{{\rm even}}(\Gamma)$. 
\end{proof}

The next result follows from Lemmas~\ref{lem1_Crticial_hyper} and \ref{lem2_Crticial_hyper}. 

\begin{cor}\label{cor1_Crticial_hyper}
Suppose that $\mathbb{H}$ is connected. 
For each nontrivial module $M$ of $\Gamma\bullet\mathbb{H}$, we have 
$M\subseteq V_1(\Gamma)$. 
\end{cor}

\begin{proof}
It follows from Lemma~\ref{lem1_Crticial_hyper} that $M/\mathfrak{C}(\Gamma)$ is a nontrivial module of $\mathbb{M}$. 
By Lemma~\ref{lem2_Crticial_hyper}, 
$M/\mathfrak{C}(\Gamma)\subseteq\mathfrak{C}_{{\rm even}}(\Gamma)$ or 
$M/\mathfrak{C}(\Gamma)\subseteq\mathfrak{C}_{{\rm odd}}(\Gamma)$. 

For a contradiction, suppose that 
$M/\mathfrak{C}(\Gamma)\subseteq\mathfrak{C}_{{\rm even}}(\Gamma)$. 
Since $\mathbb{H}$ is connected, there exists $\varepsilon\in E(\mathbb{H})$ such that 
$\varepsilon\cap(M/\mathfrak{C}(\Gamma))\neq\emptyset$ and 
$\varepsilon\setminus(M/\mathfrak{C}(\Gamma))\neq\emptyset$. 
Since $M/\mathfrak{C}(\Gamma)\subseteq\mathfrak{C}_{{\rm even}}(\Gamma)$, there exist $C\in M/\mathfrak{C}(\Gamma)$ and $D\in\mathfrak{C}_{{\rm odd}}(\Gamma)$ such that $\varepsilon=CD$. 
By Lemma~\ref{lem1_Crticial_hyper}, $M=\overline{M/\mathfrak{C}(\Gamma)}$. 
Therefore, we obtain $V(C)\subseteq M$ and $V(D)\cap M=\emptyset$. 
It follows that $V(C)$ is a nontrivial module of $(\Gamma\bullet\mathbb{H})[V(C)\cup V(D)]$, which contradicts Lemma~\ref{Cl0_Rbis_Critical_hyper}. 
Consequently, we have 
$M/\mathfrak{C}(\Gamma)\subseteq\mathfrak{C}_{{\rm odd}}(\Gamma)$. 

For a contradiction, suppose that $M\setminus V_1(\Gamma)\neq\emptyset$. 
There exists 
$D\in(M/\mathfrak{C}(\Gamma))\cap(\mathfrak{C}_{{\rm odd}}(\Gamma)\setminus\mathfrak{C}_{1}(\Gamma))$. 
Since $\mathbb{H}$ is connected, there exists $\varepsilon\in E(\mathbb{H})$ such that 
$\varepsilon\cap(M/\mathfrak{C}(\Gamma))\neq\emptyset$ and 
$\varepsilon\setminus(M/\mathfrak{C}(\Gamma))\neq\emptyset$. 
Since $M/\mathfrak{C}(\Gamma)$ is a module of $\mathbb{M}$, there exists 
$C\in M/\mathfrak{C}(\Gamma)$ such that $\varepsilon\cap(M/\mathfrak{C}(\Gamma))=\{C\}$ and 
$(\varepsilon\setminus\{C\})\cup\{D\}\in E(\mathbb{H})$. 
By Lemma~\ref{lem1_Crticial_hyper}, $V(D)\subseteq M$. 
We obtain that $V(D)$ is a nontrivial module of 
$(\Gamma\bullet\mathbb{H})[\overline{(e\setminus\{C\})\cup\{D\}}]$, which contradicts Lemma~\ref{Cl0_Rbis_Critical_hyper}. 
Consequently, $M\subseteq V_1(\Gamma)$. 
\end{proof}

\begin{lem}\label{lem3_Crticial_hyper}
Given a module $\mathbb{M}$ of $\mathbb{H}$, if 
$\mathbb{M}\subseteq\mathfrak{C}_1(\Gamma)$, then 
$\overline{\mathbb{M}}$ is a module of $\Gamma\bullet\mathbb{H}$. 
\end{lem}

\begin{proof}
Let $e\in E(\Gamma\bullet\mathbb{H})$ such that 
$e\cap\overline{\mathbb{M}}\neq\emptyset$ and 
$e\setminus\overline{\mathbb{M}}\neq\emptyset$. 
There exist $c\in e\cap\overline{\mathbb{M}}$ and 
$x\in e\setminus\overline{\mathbb{M}}$. 
We have $\{c\}\in\mathfrak{C}(\Gamma)$. 
Consider $C_x\in\mathfrak{C}(\Gamma)$ such that $x\in C_x$. 
For a contradiction, suppose that $e=cdx$, where $d\in\overline{\mathbb{M}}$. 
We have $\{d\}\in\mathfrak{C}(\Gamma)$. 
Since $cdx\in E(\Gamma\bullet\mathbb{H})$, we obtain $C_cC_dC_x\in E(\mathbb{H})$, which contradicts the fact that $\mathbb{M}$ is a module of $\mathbb{H}$. 
It follows that $e=cxy$, where 
$y\in V(\Gamma\bullet\mathbb{H})\setminus\overline{\mathbb{M}}$. 
Consider $C_y\in\mathfrak{C}(\Gamma)$ such that $y\in C_y$. 
Given $d\in\overline{\mathbb{M}}$, we have to verify that 
$dxy\in E(\Gamma\bullet\mathbb{H})$. 
We have $\{d\}\in\mathfrak{C}(\Gamma)$. 
We distinguish the following two cases. 
\begin{itemize}
\item Suppose that $C_x=C_y$.  
Since $cxy\in E(\Gamma\bullet\mathbb{H})$, we obtain 
$C_x\in\mathfrak{C}_{{\rm even}}(\Gamma)$, 
$\{x,y\}=\{\varphi_{C_x}(2i),\varphi_{C_x}(2j+1)\}$, where $0\leq i\leq j\leq w(C_x)-1$, and $\{c\}C_x\in E(\mathbb{H})$. 
Since $\mathbb{M}$ is a module of $\mathbb{H}$, $\{d\}C_x\in E(\mathbb{H})$. 
Since $C_x\in\mathfrak{C}_{{\rm even}}(\Gamma)$ and 
$\{x,y\}=\{\varphi_{C_x}(2i),\varphi_{C_x}(2j+1)\}$, where $0\leq i\leq j\leq w(C_x)-1$, we obtain $dxy\in E(H_{(\Gamma,\mathbb{H})})$. 
\item Suppose that $C_x\neq C_y$. 
Since $cxy\in E(H_{(\Gamma,\mathbb{H})})$, we obtain 
$C_x,C_y\in\mathfrak{C}_{{\rm odd}}(\Gamma)$, 
$\{x,y\}=\{\varphi_{C_x}(2i),\varphi_{C_y}(2j)\}$, where $0\leq i\leq w(C_x)$ and 
$0\leq j\leq w(C_y)$, and $\{c\}C_xC_y\in E(\mathbb{H})$. 
Since $\mathbb{M}$ is a module of $\mathbb{H}$, $\{d\}C_xC_y\in E(\mathbb{H})$. 
Since $C_x,C_y\in\mathfrak{C}_{{\rm odd}}(\Gamma)$ and 
$\{x,y\}=\{\varphi_{C_x}(2i),\varphi_{C_y}(2j)\}$, where $0\leq i\leq w(C_x)$ and 
$0\leq j\leq w(C_y)$, we obtain $dxy\in E(H_{(\Gamma,\mathbb{H})})$. \qedhere
\end{itemize}
\end{proof}

\begin{proof}[Proof of Proposition~\ref{Prop1_Critical_hyper}]
To begin, suppose that $\Gamma\bullet\mathbb{H}$ is decomposable. 
If $\mathbb{H}$ is disconnected, then Assertion~(C1) does not hold. 
Hence, suppose that $\mathbb{H}$ is connected. 
Consider a nontrivial module $M$ of $\Gamma\bullet\mathbb{H}$. 
It follows from Lemma~\ref{lem1_Crticial_hyper} that 
$M/\mathfrak{C}(\Gamma)$ is a nontrivial module of $\mathbb{H}$. 
Furthermore, $M\subseteq V_1(\Gamma)$ by Corollary~\ref{cor1_Crticial_hyper}. 
Thus, $M/\mathfrak{C}(\Gamma)\subseteq\mathfrak{C}_{1}(\Gamma)$. 
Consequently, Assertion~(C2) does not hold. 

Conversely, if $\mathbb{H}$ is disconnected, then 
it follows from Lemma~\ref{lem1bis_Crticial_hyper} that 
$\Gamma\bullet\mathbb{H}$ is decomposable. 
Hence, suppose that $\mathbb{H}$ is connected. 
Moreover, suppose that there exists a nontrivial module $\mathbb{M}$ of $\mathbb{H}$ 
such that $\mathbb{M}\subseteq\mathfrak{C}_{1}(\Gamma)$. 
It follows from Lemma~\ref{lem3_Crticial_hyper} that 
$\overline{\mathbb{M}}$ is a nontrivial module of $\Gamma\bullet\mathbb{H}$. 
Hence, $\Gamma\bullet\mathbb{H}$ is decomposable. 
\end{proof}

\section{Proof of Theorem~\ref{thm_2_main}}

In the next remark, we describe a simple way to obtain automorphisms of $\Gamma\bullet\mathbb{H}$. 

\begin{rem}\label{Auto_critical_hyper}
For every $C\in\mathfrak{C}(\Gamma)$, consider the function 
\begin{equation*}
\begin{array}{rccl}
F_C:&V(\Gamma\bullet\mathbb{H})&\longrightarrow&V(\Gamma\bullet\mathbb{H})\\
&\varphi_C(m)\ (0\leq m\leq v(C)-1)&\longmapsto&\varphi_C(v(C)-1-m)\\
&v\not\in V(C)&\longmapsto&v.
\end{array}
\end{equation*}
Clearly, $F_C$ is an automorphism of $\Gamma$. 
It is easy to verify that $F_C$ is an automorphism of $\Gamma\bullet\mathbb{H}$. 
Moreover, consider a permutation $F$ of $V(\Gamma\bullet\mathbb{H})$ such that $F(v)=v$ for each $v\in V_1(\Gamma)$. 
If for every $C\in\mathfrak{C}(\Gamma)$, $F_{\restriction V(C)}=(F_C)_{\restriction V(C)}$ or ${\rm Id}_{V(C)}$, then $F$ is an automorphism of $\Gamma\bullet\mathbb{H}$. 
\end{rem}

The next three results are useful to study the criticality of $\Gamma\bullet\mathbb{H}$. 
They allow us to prove Theorem~\ref{thm_2_main}. 

\begin{lem}\label{lem4_Crticial_hyper}
Given $C\in\mathfrak{C}_{{\rm odd}}(\Gamma)\setminus\mathfrak{C}_{1}(\Gamma)$, the following assertions hold
\begin{enumerate}
\item $V(\Gamma\bullet\mathbb{H})\setminus\{\varphi_C(0),\varphi_C(1)\}$ is a module of 
$(\Gamma\bullet\mathbb{H})-\varphi_C(0)$;
\item $V(\Gamma\bullet\mathbb{H})\setminus\{\varphi_C(2w(C)-1),\varphi_C(2w(C))\}$ is a module of 
$(\Gamma\bullet\mathbb{H})-\varphi_C(2w(C))$;
\item if $w(C)\geq 1$, then for $m\in\{1,\ldots,2w(C)-1\}$, we have 
$\{\varphi_C(m-1),\varphi_C(m+1)\}$ is a module of $(\Gamma\bullet\mathbb{H})-\varphi_C(m)$.
\end{enumerate}
\end{lem}

\begin{proof}
We begin with the first assertion. 
It follows from Remark~\ref{v(C)_odd} that $V(C)\setminus\{\varphi_C(0),\varphi_C(1)\}$ 
is a module of $((\Gamma\bullet\mathbb{H})[V(C)])-\varphi_C(0)$. 
For every $e\in E(\Gamma\bullet\mathbb{H})$, if $\varphi_C(1)\in e$, then $e\subseteq V(C)$ and $\varphi_C(0)\in e$. 
Therefore, $V(C)\setminus\{\varphi_C(0),\varphi_C(1)\}$ 
is a module of $(\Gamma\bullet\mathbb{H})-\varphi_C(0)$. 

For the second assertion, recall that the function $F_C$ is an automorphism of 
$\Gamma\bullet\mathbb{H}$ (see Remark~\ref{Auto_critical_hyper}). 
It follows from the first assertion above that $F_C(V(C)\setminus\{\varphi_C(0),\varphi_C(1)\})$ 
is a module of $(\Gamma\bullet\mathbb{H})-F_C(\varphi_C(0))$, that is, $V(C)\setminus\{\varphi_C(2w(C)-1),\varphi_C(2w(C))\}$ 
is a module of $(\Gamma\bullet\mathbb{H})-\varphi_C(2w(C))$. 

For the third assertion, suppose that $w(C)\geq 1$. 
Let $i\in\{0,\ldots,w(C)-1\}$. 
It follows from Remark~\ref{v(C)_odd} that $\{\varphi_C(2i),$ $\varphi_C(2i+2)\}$ is a module of 
$((\Gamma\bullet\mathbb{H})[V(C)])-\varphi_C(2i+1)$. 
Let $e\in E(\Gamma\bullet\mathbb{H})$ such that $e\cap\{\varphi_C(2i),\varphi_C(2i+2)\}\neq\emptyset$ and $e\setminus V(C)\neq\emptyset$. 
By definition of $\Gamma\bullet\mathbb{H}$, 
$|e\cap\{\varphi_C(2i),\varphi_C(2i+2)\}|=1$, $|e\setminus V(C)|=2$, and 
$(e\setminus V(C))\cup\{\varphi_C(2i)\},(e\setminus V(C))\cup\{\varphi_C(2i+2)\}\in E(H_{(\Gamma,\mathbb{H})})$. 
Thus, $\{\varphi_C(2i),\varphi_C(2i+2)\}$ is a module of 
$(\Gamma\bullet\mathbb{H})-\varphi_C(2i+1)$. 
Lastly, suppose that $w(C)\geq 2$, and consider $i\in\{1,\ldots,w(C)-1\}$. 
It follows from Remark~\ref{v(C)_odd} that 
$\{\varphi_C(2i-1),\varphi_C(2i+1)\}$ is a module of 
$((\Gamma\bullet\mathbb{H})[V(C)])-\varphi_C(2i)$. 
For every $e\in E(\Gamma\bullet\mathbb{H})$, if $e\cap\{\varphi_C(2i-1),\varphi_C(2i+1)\}\neq\emptyset$, then $e\subseteq V(C)$. 
Therefore,  $\{\varphi_C(2i-1),\varphi_C(2i+1)\}$ is a module of 
$(\Gamma\bullet\mathbb{H})-\varphi_C(2i)$. 
\end{proof}

\begin{lem}\label{lem5_Crticial_hyper}
Given $C\in\mathfrak{C}_{{\rm even}}(\Gamma)$, the following assertions hold
\begin{enumerate}
\item $V(\Gamma\bullet\mathbb{H})\setminus\{\varphi_C(0),\varphi_C(1)\}$ is a module of 
$(\Gamma\bullet\mathbb{H})-\varphi_C(0)$;
\item $V(\Gamma\bullet\mathbb{H})\setminus\{\varphi_C(2w(C)-2),\varphi_C(2w(C)-1)\}$ is a module of $(\Gamma\bullet\mathbb{H})-\varphi_C(2w(C)-1)$;
\item if $w(C)\geq 2$, then for $m\in\{1,\ldots,2w(C)-2\}$, we have 
$\{\varphi_C(m-1),\varphi_C(m+1)\}$ is a module of $(\Gamma\bullet\mathbb{H})-\varphi_C(m)$.
\end{enumerate}
\end{lem}

\begin{proof}
For the first assertion, consider $e\in E(\Gamma\bullet\mathbb{H})$ such that 
$\varphi_C(1)\in e$. 
By definition of $\Gamma\bullet\mathbb{H}$, there exists 
$D\in\mathfrak{C}_{{\rm odd}}(\Gamma)$ such that 
$e=\varphi_C(0)\varphi_C(1)\varphi_D(2k)$, where $0\leq k\leq w(D)$. 
Consequently, there does not exist $e\in E((\Gamma\bullet\mathbb{H})-\varphi_C(0))$ such that $\varphi_C(1)\in e$. 
It follows that $V(\Gamma\bullet\mathbb{H})\setminus\{\varphi_C(0),\varphi_C(1)\}$ is a module of $(\Gamma\bullet\mathbb{H})-\varphi_C(0)$. 

As in the proof of the second assertion of Lemma~\ref{lem4_Crticial_hyper}, the second assertion is deduced from the first one by using Remark~\ref{Auto_critical_hyper}. 

For the third assertion, suppose that $w(C)\geq 2$. 
Consider $i\in\{1,\ldots,w(C)-1\}$. 
We prove that $\{\varphi_C(2i-1),\varphi_C(2i+1)\}$ is a module of 
$(\Gamma\bullet\mathbb{H})-\varphi_C(2i)$. 
Let $e\in E((\Gamma\bullet\mathbb{H})-\varphi_C(2i))$ such that 
$e\cap\{\varphi_C(2i-1),\varphi_C(2i+1)\}\neq\emptyset$ and 
$e\setminus\{\varphi_C(2i-1),\varphi_C(2i+1)\}\neq\emptyset$. 
By definition of $\Gamma\bullet\mathbb{H}$, there exist 
$D\in\mathfrak{C}_{{\rm odd}}(\Gamma)$ and $k\in\{0,\ldots,w(D)\}$ such that 
\begin{equation*}
\begin{cases}
\text{$e=\varphi_C(2j)\varphi_C(2i-1)\varphi_D(2k)$, where $j\in\{0,\ldots,i-1\}$}\\
\text{or}\\
\text{$e=\varphi_C(2j)\varphi_C(2i+1)\varphi_D(2k)$, where $j\in\{0,\ldots,i\}$}. 
\end{cases}
\end{equation*}
In the second instance, we have $j\in\{0,\ldots,i-1\}$ because $\varphi_C(2i)\not\in e$. 
By definition of $\Gamma\bullet\mathbb{H}$, we have 
$\varphi_C(2j)\varphi_C(2i-1)\varphi_D(2k),\varphi_C(2j)\varphi_C(2i+1)\varphi_D(2k)\in 
E((\Gamma\bullet\mathbb{H})-\varphi_C(2i))$. 
Thus, $\{\varphi_C(2i-1),\varphi_C(2i+1)\}$ is a module of 
$(\Gamma\bullet\mathbb{H})-\varphi_C(2i)$. 
Lastly, consider $i\in\{0,\ldots,w(C)-2\}$. 
We prove that $\{\varphi_C(2i),\varphi_C(2i+2)\}$ is a module of 
$(\Gamma\bullet\mathbb{H})-\varphi_C(2i+1)$. 
Let $e\in E((\Gamma\bullet\mathbb{H})-\varphi_C(2i+1))$ such that 
$e\cap\{\varphi_C(2i),\varphi_C(2i+2)\}\neq\emptyset$ and 
$e\setminus\{\varphi_C(2i),\varphi_C(2i+2)\}\neq\emptyset$. 
By definition of $\Gamma\bullet\mathbb{H}$, there exist 
$D\in\mathfrak{C}_{{\rm odd}}(\Gamma)$ and $k\in\{0,\ldots,w(D)\}$ such that 
\begin{equation*}
\begin{cases}
\text{$e=\varphi_C(2i)\varphi_C(2j+1)\varphi_D(2k)$, where $j\in\{i,\ldots,w(C)-1\}$}\\
\text{or}\\
\text{$e=\varphi_C(2i+2)\varphi_C(2j+1)\varphi_D(2k)$, where $j\in\{i+1,\ldots,w(C)-1\}$}. 
\end{cases}
\end{equation*}
In the first instance, we have $j\in\{i+1,\ldots,w(C)-1\}$ because $\varphi_C(2i+1)\not\in e$. 
By definition of $\Gamma\bullet\mathbb{H}$, we have 
$\varphi_C(2i)\varphi_C(2j+1)\varphi_D(2k),\varphi_C(2i+2)\varphi_C(2j+1)\varphi_D(2k)\in 
E((\Gamma\bullet\mathbb{H})-\varphi_C(2i+1))$. 
Thus, $\{\varphi_C(2i),\varphi_C(2i+2)\}$ is a module of 
$(\Gamma\bullet\mathbb{H})-\varphi_C(2i+1)$. 
\end{proof}

The next result is an immediate consequence of Lemmas~\ref{lem4_Crticial_hyper} 
and \ref{lem5_Crticial_hyper}. 

\begin{cor}\label{cor2_Crticial_hyper}
Let $C\in\mathfrak{C}(\Gamma)\setminus\mathfrak{C}_{1}(\Gamma)$. 
For every $c\in V(C)$, we have $(\Gamma\bullet\mathbb{H})-c$ is decomposable. 
Moreover, we have 
\begin{enumerate}
\item $N_{\mathscr{P}(\Gamma\bullet\mathbb{H})}(\varphi_C(0))\subseteq\{\varphi_C(1)\}$;
\item $N_{\mathscr{P}(\Gamma\bullet\mathbb{H})}(\varphi_C(v(C)-1))\subseteq\{\varphi_C(v(C)-2)\}$;
\item if $v(C)\geq 3$, then for every $m\in\{1,\ldots,v(C)-2\}$, we have 
$N_{\mathscr{P}(\Gamma\bullet\mathbb{H})}(\varphi_C(m))\subseteq\{\varphi_C(m-1),\varphi_C(m+1)\}$. 
\end{enumerate}
\end{cor}

\begin{proof}[Proof of Theorem~\ref{thm_2_main}]
To begin, suppose that $\Gamma\bullet\mathbb{H}$ is critical. 
In particular, $\Gamma\bullet\mathbb{H}$ is prime. 
By Proposition~\ref{Prop1_Critical_hyper}, Assertions (C1) and (C2) hold. 
For Assertion (C3), consider $v\in V_1(\Gamma)$. 
Since $\Gamma\bullet\mathbb{H}$ is critical, $(\Gamma\bullet\mathbb{H})-v$ is decomposable. 
Furthermore, since $v\in V_1(\Gamma)$, we have $\{v\}\in\mathfrak{C}_{1}(\Gamma)$. 
It follows that $(\Gamma\bullet\mathbb{H})-v=(\Gamma-v)\bullet(\mathbb{H}-\{v\})$. 
By Proposition~\ref{Prop1_Critical_hyper} applied to $(\Gamma\bullet\mathbb{H})-v$, $\mathbb{H}-\{v\}$ is disconnected or 
$\mathbb{H}-\{v\}$ is connected, and there exists a nontrivial module 
$\mathbb{M}_{\{v\}}$ of $\mathbb{H}-\{v\}$ such that 
$\mathbb{M}_{\{v\}}\subseteq\mathfrak{C}_1(\Gamma-v)$. 
Since $\mathfrak{C}_1(\Gamma-v)=\mathfrak{C}_1(\Gamma)\setminus\{\{v\}\}$, 
$\mathbb{M}_{\{v\}}\subseteq\mathfrak{C}_1(\Gamma)\setminus\{\{v\}\}$. 

Conversely, suppose that Assertions (C1), (C2) and (C3) hold. 
Since Assertions (C1) and (C2) hold, it follows from Proposition~\ref{Prop1_Critical_hyper}, 
$\Gamma\bullet\mathbb{H}$ is prime. 
Furthermore, it follows from Assertion (C3) and Proposition~\ref{Prop1_Critical_hyper} that 
$(\Gamma\bullet\mathbb{H})-v$ is decomposable for each 
$v\in V_1(\Gamma)$. 
Lastly, consider $v\in V(\Gamma)\setminus V_1(\Gamma)$. 
There exists $C\in\mathfrak{C}(\Gamma)\setminus\mathfrak{C}_{1}(\Gamma)$ such that $v\in V(C)$. 
By Corollary~\ref{cor2_Crticial_hyper}, $(\Gamma\bullet\mathbb{H})-v$ is decomposable. 
Consequently, $\Gamma\bullet\mathbb{H}$ is critical. 
\end{proof}

\section{An improvement of Theorem~\ref{thm_2_main}}\label{sec_improvement}

The purpose of this section is to demonstrate the following result. 

\begin{thm}\label{thm_3_main}
Suppose that $v(\Gamma\bullet\mathbb{H})\geq 5$. 
The 3-hypergraph $\Gamma\bullet\mathbb{H}$ is critical and 
$\mathscr{P}(\Gamma\bullet\mathbb{H})=\Gamma$ 
if and only if the following six assertions hold
\begin{enumerate}
\item[(C1)] $\mathbb{H}$ is connected;
\item[(C2)] for every nontrivial module $\mathbb{M}$ of $\mathbb{H}$, we have 
$\mathbb{M}\setminus\mathfrak{C}_{1}(\Gamma)\neq\emptyset$;
\item[(C3)] for each $v\in V_1(\Gamma)$, if 
$\mathbb{H}-\{v\}$ is connected, then $\mathbb{H}-\{v\}$ admits a nontrivial module $\mathbb{M}_{\{v\}}$ such that $\mathbb{M}_{\{v\}}\subseteq\mathfrak{C}_1(\Gamma)\setminus\{\{v\}\}$;
\item[(C4)] for each $C\in\mathfrak{C}_{{\rm even}}(\Gamma)$ such that $v(C)=2$, we have 
\begin{itemize}
\item $\mathbb{H}-C$ is connected,
\item for every nontrivial module $\mathbb{M}_C$ of $\mathbb{H}-C$, 
we have $\mathbb{M}_C\setminus\mathfrak{C}_{1}(\Gamma)\neq\emptyset$;
\end{itemize}
\item[(C5)] for each $C\in\mathfrak{C}_{{\rm odd}}(\Gamma)$ such that $v(C)=3$, we have 
\begin{itemize}
\item $\mathbb{H}_C$ is connected, where $\mathbb{H}_C$ is obtained from $\mathbb{H}$ by replacing $C$ by $\{\varphi_C(0)\}$, 
\item for every nontrivial module $\mathbb{M}_C$ of $\mathbb{H}_C$, 
we have $\mathbb{M}_C\setminus(\mathfrak{C}_{1}(\Gamma)\cup\{\varphi_C(0)\})\neq\emptyset$; 
\end{itemize}
\item[(C6)] for each $C\in\mathfrak{C}_1(\Gamma)$, if there exists $D\in\mathfrak{C}_1(\Gamma)\setminus\{C\}$ such that $D$ is an isolated vertex of $\mathbb{H}-C$ and 
$\mathbb{H}\setminus\{C,D\}$ is connected, then $\mathbb{H}\setminus\{C,D\}$ admits a nontrivial module $\mathbb{M}$ such that 
$\mathbb{M}\subseteq\mathfrak{C}_1(\Gamma)\setminus\{C,D\}$. 
\end{enumerate}
\end{thm}

We use the next four results to prove Theorem~\ref{thm_3_main}. 

\begin{prop}\label{Prop2_Critical_hyper}
Suppose that $v(\Gamma\bullet\mathbb{H})\geq 5$. 
If the 3-hypergraph $\Gamma\bullet\mathbb{H}$ is critical, then 
for each $C\in\mathfrak{C}(\Gamma)$ such that $v(C)\geq 4$, we have 
\begin{equation}\label{E1bis_thm1bis_Critical_hyper}
\begin{cases}
\mathscr{P}(\Gamma\bullet\mathbb{H})[V(C)]=C,\\
\text{and}\\
\text{for each $c\in V(C)$, 
$N_{\mathscr{P}(\Gamma\bullet\mathbb{H})}(c)\subseteq V(C).$}
\end{cases}
\end{equation} 
\end{prop}

\begin{proof}
By Corollary~\ref{cor2_Crticial_hyper}, for each $C\in\mathfrak{C}(\Gamma)\setminus\mathfrak{C}_{1}(\Gamma)$ and for every $c\in V(C)$, we have 
$N_{\mathscr{P}(\Gamma\bullet\mathbb{H})}(c)\subseteq V(C)$. 
Now, consider $C\in\mathfrak{C}(\Gamma)$ such that $v(C)\geq 4$. 
For instance, suppose that $C\in\mathfrak{C}_{{\rm odd}}(\Gamma)$. 
We have $w(C)\geq 2$. 
Consider $i\in\{1,\ldots,w(C)-1\}$. 
We verify that 
\begin{equation}\label{E2_thm1_Critical_hyper}
N_{\mathscr{P}(\Gamma\bullet\mathbb{H})}(\varphi_C(2i))=
\{\varphi_C(2i-1),\varphi_C(2i+1)\}. 
\end{equation}
We show that $(\Gamma\bullet\mathbb{H})-\{\varphi_C(2i-1),\varphi_C(2i)\}$ is prime. 
To define $\Gamma\bullet\mathbb{H}$, we consider an isomorphism from $P_{2w(C)+1}$ onto $C$. 
It follows from the definition of $\Gamma\bullet\mathbb{H}$ that $\varphi_C$ is an isomorphism from 
$C_3(U_{2w(C)+1})$ onto $(\Gamma\bullet\mathbb{H})[V(C)]$. 
The bijection 
\begin{equation*}
\begin{array}{rccl}
\psi:&\{0,\ldots,2w(C)-2\}&\longrightarrow&\{0,\ldots,2w(C)\}\setminus\{2i-1,2i\}\\
&0\leq j\leq 2i-2&\longmapsto&j\\
&2i-1\leq j\leq 2w(C)-1&\longmapsto&j+2,
\end{array}
\end{equation*}
is an isomorphism from $C_3(U_{2w(C)-1})$ onto $C_3(U_{2w(C)+1})-\{2i-1,2i\}$. 
Therefore, $((\varphi_C)_{\restriction\{0,\ldots,2w(C)+1\}\setminus\{i-1,i\}})\circ\psi$ is an isomorphism from 
$C_3(U_{2w(C)-1})$ onto $(\Gamma\bullet\mathbb{H})[V(C)\setminus\{\varphi_C(2i-1),\varphi_C(2i)\}]$. 
Observe that $\psi$ is also an isomorphism from $P_{2w(C)-1}$ onto the path $C^-$ defined on 
$V(C^-)=\{0,\ldots,2w(C)+1\}\setminus\{i-1,i\}$ by 
$E(C^-)=(E(C)\setminus\{\varphi_C(2i-2)\varphi_C(2i-1),\varphi_C(2i-1)\varphi_C(2i),
\varphi_C(2i)\varphi_C(2i+1)\})\cup\{\varphi_C(2i-2)\varphi_C(2i+1)\}$. 
Now, consider the graph $\Gamma^-$ obtained from $\Gamma$ by replacing the path $C$ by $C^-$. 
Moreover, consider the hypergraph $\mathbb{H}^-$ obtained from 
$\mathbb{H}$ by renaming the vertex $C$ by $C^-$. 
Hence, the bijection 
\begin{equation*}
\begin{array}{rrcl}
\theta:&V(\mathbb{H})&\longrightarrow&(V(\mathbb{H})\setminus\{C\})\cup\{C^-\}\\
&D\neq C&\longmapsto&D\\
&C&\longmapsto&C^-,
\end{array}
\end{equation*}
is an isomorphism from $\mathbb{H}$ onto $\mathbb{H}^-$. 
Since $\psi$ is increasing and preserves the parity, we obtain 
$(\Gamma\bullet\mathbb{H})-\{\varphi_C(2i-1),\varphi_C(2i)\}=
\Gamma^-\bullet\mathbb{H}^-$. 
Since $\mathbb{H}$ is connected, $\mathbb{H}^-$ is as well. 
Furthermore, consider any nontrivial module $\mathbb{M}^-$ of 
$\mathbb{H}^-$. 
We obtain that $\theta^{-1}(\mathbb{M}^-)$ is a nontrivial module of $\mathbb{H}$. 
Since $\Gamma\bullet\mathbb{H}$ is critical, it is prime. 
By Assertion~(C2) of Proposition~\ref{Prop1_Critical_hyper}, 
$\theta^{-1}(\mathbb{M}^-)\setminus\mathfrak{C}_{1}(\Gamma)\neq\emptyset$. 
Therefore, we obtain 
$(\mathbb{M}^-)\setminus\theta(\mathfrak{C}_{1}(\Gamma))\neq\emptyset$. 
Clearly, we have $\theta(\mathfrak{C}_{1}(\Gamma))=
\mathfrak{C}_{1}(\Gamma)=\mathfrak{C}_{1}(\Gamma^-)$. 
It follows that 
$(\mathbb{M}^-)\setminus\mathfrak{C}_{1}(\Gamma^-)\neq\emptyset$. 
By Proposition~\ref{Prop1_Critical_hyper}, 
$\Gamma^-\bullet\mathbb{H}^-$, that is, 
$(\Gamma\bullet\mathbb{H})-\{\varphi_C(2i-1),\varphi_C(2i)\}$ is prime. 
Consequently, 
$\varphi_C(2i-1)\in N_{\mathscr{P}(\Gamma\bullet\mathbb{H})}(\varphi_C(2i))$.  
Similarly, $\varphi_C(2i+1)\in N_{\mathscr{P}(\Gamma\bullet\mathbb{H})}(\varphi_C(2i))$. 
Therefore, \eqref{E2_thm1_Critical_hyper} follows from Lemma~\ref{Lcritical}. 

Now, we verify that 
\begin{equation}\label{E3_thm1_Critical_hyper}
N_{\mathscr{P}(H_{(\Gamma,\mathbb{H})})}(\varphi_C(0))=
\{\varphi_C(1)\}. 
\end{equation}
As previously, we consider the graph $C^-$ defined on $V(C)\setminus\{\varphi_C(0),\varphi_C(1)\}$ by 
$E(C^-)=E(C)\setminus\{\varphi_C(0)\varphi_C(1),\varphi_C(1)\varphi_C(2)\}$. 
Now, consider the graph $\Gamma^-$ obtained from $\Gamma$ by replacing the path $C$ by $C^-$. 
Moreover, consider the hypergraph $\mathbb{H}^-$ obtained from 
$\mathbb{H}$ by renaming the vertex $C$ by $C^-$. 
We obtain 
$(\Gamma\bullet\mathbb{H})-\{\varphi_C(0),\varphi_C(1)\}=
\Gamma^-\bullet\mathbb{H}^-$. 
It follows from Proposition~\ref{Prop1_Critical_hyper} applied with 
$\Gamma^-$ and $\mathbb{H}^-$ that 
$(\Gamma\bullet\mathbb{H})-\{\varphi_C(0),\varphi_C(1)\}$ is prime. 
Hence, 
$\varphi_C(1)\in N_{\mathscr{P}(\Gamma\bullet\mathbb{H})}(\varphi_C(0))$. 
Set $X=V(\Gamma\bullet\mathbb{H})\setminus\{\varphi_C(0),\varphi_C(1)\}$. 
Since $(\Gamma\bullet\mathbb{H})[X]$ is prime and 
$\varphi_C(1)\in\langle X\rangle_{\Gamma\bullet\mathbb{H}}$, it follows from 
Lemma~\ref{Lcritical} that \eqref{E3_thm1_Critical_hyper} holds. 
It follows from Remark~\ref{Auto_critical_hyper} that 
\begin{equation}\label{E4_thm1_Critical_hyper}
N_{\mathscr{P}(\Gamma\bullet\mathbb{H})}(\varphi_C(2w(C)))=
\{\varphi_C(2w(C)-1)\}. 
\end{equation}

Lastly, we prove that 
\begin{equation}\label{E5_thm1_Critical_hyper}
N_{\mathscr{P}(\Gamma\bullet\mathbb{H})}(\varphi_C(2i+1))=
\{\varphi_C(2i),\varphi_C(2i+2)\}
\end{equation}
for each $i\in\{0,\ldots,w(C)-1\}$. 
It follows from \eqref{E2_thm1_Critical_hyper} that 
$\varphi_C(2i)\in N_{\mathscr{P}(\Gamma\bullet\mathbb{H})}(\varphi_C(2w(C)+1))$ and 
$\varphi_C(2i+2)\in N_{\mathscr{P}(\Gamma\bullet\mathbb{H})}(\varphi_C(2w(C)+1))$. 
By Lemma~\ref{Lcritical}, \eqref{E5_thm1_Critical_hyper} holds. 

Consequently, it follows from \eqref{E2_thm1_Critical_hyper}, \eqref{E3_thm1_Critical_hyper}, \eqref{E4_thm1_Critical_hyper}, \eqref{E5_thm1_Critical_hyper}, and from Lemma~\ref{Lcritical} that 
$\mathscr{P}(\Gamma\bullet\mathbb{H})[V(C)]=C$. 
\end{proof}

\begin{lem}\label{lem_6_Critical_hyper}
Suppose that $\Gamma\bullet\mathbb{H}$ is critical with $v(\Gamma\bullet\mathbb{H})\geq 5$. 
Given $C\in\mathfrak{C}(\Gamma)$ such that $v(C)=2$, $C$ satisfies \eqref{E1bis_thm1bis_Critical_hyper} if and only if the following two assertions hold 
\begin{itemize}
\item $\mathbb{H}-C$ is connected;
\item for every nontrivial module $\mathbb{M}_C$ of $\mathbb{H}-C$, 
we have $\mathbb{M}_C\setminus\mathfrak{C}_{1}(\Gamma)\neq\emptyset$;
\end{itemize}
\end{lem}

\begin{proof}
By Corollary~\ref{cor2_Crticial_hyper}, for every $c\in V(C)$, we have 
$N_{\mathscr{P}(\Gamma\bullet\mathbb{H})}(c)\subseteq V(C)$. 
Consider an isomorphism $\varphi_C$ from $P_2$ onto $C$. 
We have $(\Gamma\bullet\mathbb{H})-\{\varphi_C(0),\varphi_C(1)\}=
(\Gamma-V(C))\bullet(\mathbb{H}-C)$. 
Moreover, we have 
$\mathfrak{C}_{1}(\Gamma-V(C))=\mathfrak{C}_{1}(\Gamma)$. 
To conclude, it suffices to apply Proposition~\ref{Prop1_Critical_hyper} to 
$(\Gamma-V(C))\bullet(\mathbb{H}-C)$. 
\end{proof}

\begin{lem}\label{lem_7_Critical_hyper}
Suppose that $\Gamma\bullet\mathbb{H})$ is critical with 
$v(\Gamma\bullet\mathbb{H}))\geq 5$. 
Given $C\in\mathfrak{C}(\Gamma)$ such that $v(C)=3$, $C$ satisfies \eqref{E1bis_thm1bis_Critical_hyper} if and only if the following two assertions hold 
\begin{itemize}
\item $\mathbb{H}_C$ is connected, where $\mathbb{H}_C$ is obtained from $\mathbb{H}$ by replacing $C$ by $\{\varphi_C(0)\}$;
\item for every nontrivial module $\mathbb{M}_C$ of $\mathbb{H}_C$, 
we have $\mathbb{M}_C\setminus(\mathfrak{C}_{1}(\Gamma)\cup\{\varphi_C(0)\})\neq\emptyset$. 
\end{itemize}
\end{lem}

\begin{proof}
By Corollary~\ref{cor2_Crticial_hyper}, for every $c\in V(C)$, we have 
$N_{\mathscr{P}(\Gamma\bullet\mathbb{H})}(c)\subseteq V(C)$. 
Furthermore, consider an isomorphism $\varphi_C$ from $P_3$ onto $C$. 
Denote by $\mathbb{H}_C$ the $\{2,3\}$-hypergraph obtained from $\mathbb{H}$ by replacing $C$ by 
$\{\varphi_C(0)\}$. 
We have 
\begin{equation}\label{E1_lem_7_Critical_hyper}
\begin{cases}
(\Gamma\bullet\mathbb{H})-\{\varphi_C(1),\varphi_C(2)\}=
(\Gamma-\{\varphi_C(1),\varphi_C(2)\})\bullet\mathbb{H}_C\\
\text{and}\\
\mathfrak{C}_{1}(\Gamma-\{\varphi_C(1),\varphi_C(2)\})=\mathfrak{C}_{1}(\Gamma)\cup\{\varphi_C(0)\}.
\end{cases}
\end{equation}

To begin, suppose that $C$ satisfies \eqref{E1bis_thm1bis_Critical_hyper}. 
We obtain that $(\Gamma\bullet\mathbb{H})-\{\varphi_C(1),$ $\varphi_C(2)\}$ is prime. 
By Proposition~\ref{Prop1_Critical_hyper} applied to 
$H_{(\Gamma-\{\varphi_C(1),\varphi_C(2)\},\mathbb{H}_C)}$, both assertions above hold. 

Conversely, suppose that both assertions above hold. 
Since \eqref{E1_lem_7_Critical_hyper} holds, it follows from Proposition~\ref{Prop1_Critical_hyper} that 
$(\Gamma\bullet\mathbb{H})-\{\varphi_C(1),\varphi_C(2)\}$ is prime. 
By Corollary~\ref{cor2_Crticial_hyper}, $\{\varphi_C(0),\varphi_C(2)\}$ is a module of 
$(\Gamma\bullet\mathbb{H})-\varphi_C(1)$. 
Therefore, $(\Gamma\bullet\mathbb{H})-\{\varphi_C(0),\varphi_C(1)\}$ and 
$(\Gamma\bullet\mathbb{H})-\{\varphi_C(1),\varphi_C(2)\}$ are isomorphic. 
It follows that $(\Gamma\bullet\mathbb{H})-\{\varphi_C(0),\varphi_C(1)\}$ is prime. 
Lastly, it follows from Corollary~\ref{cor2_Crticial_hyper} that 
$V(\Gamma\bullet\mathbb{H})\setminus\{\varphi_C(0),\varphi_C(1)\}$ is a module of 
$(\Gamma\bullet\mathbb{H})-\varphi_C(0)$. 
Thus, $V(\Gamma\bullet\mathbb{H})\setminus\{\varphi_C(0),\varphi_C(1),\varphi_C(2)\}$ is a non trivial module of $(\Gamma\bullet\mathbb{H})-\{\varphi_C(0),\varphi_C(2)\}$. 
Hence, $(\Gamma\bullet\mathbb{H})-\{\varphi_C(0),\varphi_C(2)\}$ is decomposable. 
It follows that 
$\mathscr{P}(\Gamma\bullet\mathbb{H})[$ $V(C)]=C$. 
It follows that $C$ satisfies \eqref{E1bis_thm1bis_Critical_hyper}. 
\end{proof}

\begin{lem}\label{lem_8_Critical_hyper}
Suppose that $\Gamma\bullet\mathbb{H}$ is critical with $v(\Gamma\bullet\mathbb{H})\geq 5$. 
For every $v\in V_1(\Gamma)$, we have $d_{\mathscr{P}(\Gamma\bullet\mathbb{H})}=1$ if and only if there exists $w\in V_1(\Gamma)$ satisfying 
\begin{itemize}
\item $\{w\}$ is an isolated vertex of $\mathbb{H}-\{v\}$;
\item $\mathbb{H}-\{\{v\},\{w\}\}$ is connected;
\item for every nontrivial module $\mathbb{M}_{vw}$ of $\mathbb{H}-\{\{v\},\{w\}\}$, 
we have $\mathbb{M}_{vw}\setminus(\mathfrak{C}_{1}(\Gamma)\setminus\{\{v\},\{w\}\})\neq\emptyset$. 
\end{itemize}
\end{lem}

\begin{proof}
To begin, suppose that $d_{\mathscr{P}(\Gamma\bullet\mathbb{H})}(v)=1$. 
Denote by $w$ the unique neighbour of $v$ in $\mathscr{P}(\Gamma\bullet\mathbb{H})$. 
It follows from Proposition~\ref{Prop2_Critical_hyper}, Lemma~\ref{lem_6_Critical_hyper}, and 
Lemma~\ref{lem_7_Critical_hyper} that $w\in V_1(\Gamma)$. 
We have 
\begin{equation}\label{E1_lem_8_Critical_hyper}
\begin{cases}
(\Gamma\bullet\mathbb{H})-\{v,w\}=
(\Gamma-\{v,w\})\bullet(\mathbb{H}-\{\{v\},\{w\}\})\\
\text{and}\\
\mathfrak{C}_{1}(\Gamma-\{v,w\})=\mathfrak{C}_{1}(\Gamma)\setminus\{\{v\},\{w\}\}.
\end{cases}
\end{equation}
Since $N_{\mathscr{P}(\Gamma\bullet\mathbb{H})}(v)=\{w\}$, $(\Gamma\bullet\mathbb{H})-\{v,w\}$ is prime. 
By Proposition~\ref{Prop1_Critical_hyper} applied to 
$(\Gamma-\{v,w\})\bullet(\mathbb{H}-\{\{v\},\{w\}\})$, 
$\mathbb{H}-\{\{v\},\{w\}\}$ is connected, and 
for every nontrivial module $\mathbb{M}_{vw}$ of $\mathbb{H}-\{\{v\},\{w\}\}$, 
$\mathbb{M}_{vw}\setminus(\mathfrak{C}_{1}(\Gamma)\setminus\{\{v\},\{w\}\})\neq\emptyset$. 
Moreover, since $N_{\mathscr{P}(\Gamma\bullet\mathbb{H})}(v)=\{w\}$, it follows from Lemma~\ref{Lcritical} that $V(\Gamma\bullet\mathbb{H})\setminus\{v,w\}$ is a module of $(\Gamma\bullet\mathbb{H})-v$. 
We show that 
\begin{equation}\label{E2_lem_8_Critical_hyper}
\text{for each $\varepsilon\in E(\mathbb{H})$, if $\{w\}\in \varepsilon$, then $\{v\}\in \varepsilon$.}
\end{equation}
Indeed, consider $\varepsilon\in E(\mathbb{H})$ such that $\{w\}\in \varepsilon$. 
For a contradiction, suppose that $|\varepsilon|=2$. 
There exists $C\in\mathfrak{C}_{{\rm even}}(\Gamma)$ such that $\varepsilon=\{w\}C$.  
We obtain $C\subseteq V(\Gamma\bullet\mathbb{H})\setminus\{v,w\}$ and 
$w\varphi_C(0)\varphi_C(1)\in E(\Gamma\bullet\mathbb{H})$, which contradicts the fact that 
$V(\Gamma\bullet\mathbb{H})\setminus\{v,w\}$ is a module of $(\Gamma\bullet\mathbb{H})-v$. 
Consequently, we have $|\varepsilon|=3$. 
Thus, there exist distinct $C,D\in\mathfrak{C}_{{\rm odd}}(\Gamma)$ such that $\varepsilon=\{w\}CD$.  
If $C\neq\{v\}$ and $D\neq\{v\}$, then 
$\varphi_C(0),\varphi_D(0)\in V(\Gamma\bullet\mathbb{H})\setminus\{v,w\}$ and 
$w\varphi_C(0)\varphi_D(0)\in E(\Gamma\bullet\mathbb{H})$, which contradicts the fact that 
$V(\Gamma\bullet\mathbb{H})\setminus\{v,w\}$ is a module of $(\Gamma\bullet\mathbb{H})-v$. 
Therefore, $C=\{v\}$ or $D=\{v\}$. 
Hence, \eqref{E2_lem_8_Critical_hyper} holds. 
It follows that $\{w\}$ is an isolated vertex of $\mathbb{H}-\{v\}$. 

Conversely, suppose that there exists $w\in V_1(\Gamma)$ such that the three assertions above hold. 
Since \eqref{E1_lem_8_Critical_hyper} holds, it follows from Proposition~\ref{Prop1_Critical_hyper} applied to 
$(\Gamma-\{v,w\})\bullet(\mathbb{H}-\{\{v\},\{w\}\})$ 
that $(\Gamma\bullet\mathbb{H})-\{v,w\}$ is prime. 
Thus, we have 
\begin{equation}\label{E3_lem_8_Critical_hyper}
w\in N_{\mathscr{P}(\Gamma\bullet\mathbb{H})}(v).
\end{equation}
Furthermore, since $\{w\}$ is an isolated vertex of $\mathbb{H}-\{v\}$, $\mathbb{H}[\{\{w\}\}]$ 
is a component of $\mathbb{H}-\{v\}$. 
By Lemma~\ref{lem1bis_Crticial_hyper}, 
$V(\Gamma\bullet\mathbb{H})\setminus\{v,w\}$ is a module of $(\Gamma\bullet\mathbb{H})-v$. 
Consider any $u\in V(\Gamma\bullet\mathbb{H})\setminus\{v,w\}$. 
We obtain that 
$V(\Gamma\bullet\mathbb{H})\setminus\{u,v,w\}$ is a module of $(\Gamma\bullet\mathbb{H})-\{u,v\}$. 
Hence, we obtain $u\not\in N_{\mathscr{P}(\Gamma\bullet\mathbb{H})}(v)$. 
It follows from \eqref{E3_lem_8_Critical_hyper} that 
$N_{\mathscr{P}(\Gamma\bullet\mathbb{H})}(v)=\{w\}$. 
\end{proof}

\begin{proof}[Proof of Theorem~\ref{thm_3_main}]
To begin, suppose that $\Gamma\bullet\mathbb{H}$ is critical and 
$\mathscr{P}(\Gamma\bullet\mathbb{H})=\Gamma$. 
Since $\Gamma\bullet\mathbb{H}$ is critical, it follows from Theorem~\ref{thm_2_main} that Assertions (C1), (C2), and (C3) hold. 
For Assertion (C4), consider $C\in\mathfrak{C}_{{\rm even}}(\Gamma)$ such that $v(C)=2$. 
By Corollary~\ref{cor2_Crticial_hyper}, for every $c\in V(C)$, we have 
$N_{\mathscr{P}(\Gamma\bullet\mathbb{H})}(c)\subseteq V(C)$. 
Therefore, since $\mathscr{P}(\Gamma\bullet\mathbb{H})=\Gamma$, 
$C$ satisfies \eqref{E1bis_thm1bis_Critical_hyper}. 
By Lemma~\ref{lem_6_Critical_hyper}, $C$ satisfies Assertion (C4). 

For Assertion (C5), consider $C\in\mathfrak{C}_{{\rm odd}}(\Gamma)$ 
such that $v(C)=3$. 
By Corollary~\ref{cor2_Crticial_hyper}, for every $c\in V(C)$, we have 
$N_{\mathscr{P}(\Gamma\bullet\mathbb{H})}(c)\subseteq V(C)$. 
Therefore, since $\mathscr{P}(\Gamma\bullet\mathbb{H})=\Gamma$, 
$C$ satisfies \eqref{E1bis_thm1bis_Critical_hyper}. 
By Lemma~\ref{lem_7_Critical_hyper}, $C$ satisfies Assertion (C5).

For Assertion (C6), consider distinct 
$C,D\in\mathfrak{C}_1(\Gamma)$ such that $D$ is an isolated vertex of $\mathbb{H}-C$ and $\mathbb{H}\setminus\{C,D\}$ is connected. 
Since $\mathscr{P}(\Gamma\bullet\mathbb{H})=\Gamma$, we have 
$d_{\mathscr{P}(\Gamma\bullet\mathbb{H})}(C)=0$. 
It follows from Lemma~\ref{lem_8_Critical_hyper} that there exists a 
nontrivial module $\mathbb{M}$ of $\mathbb{H}-\{C,D\}$ such that 
$\mathbb{M}\subseteq\mathfrak{C}_{1}(\Gamma)\setminus\{C,D\}$. 

Conversely, suppose that Assertions (C1),...,(C6) hold. 
By Theorem~\ref{thm_2_main}, $\Gamma\bullet\mathbb{H}$ is critical. 
It remains to show that 
\begin{equation}\label{E1_thm2_Critical_hyper}
\mathscr{P}(\Gamma\bullet\mathbb{H})=\Gamma.
\end{equation}
It follows from Proposition~\ref{Prop2_Critical_hyper} that 
for each $C\in\mathfrak{C}(\Gamma)$ such that $v(C)\geq 4$, $C$ satisfies \eqref{E1bis_thm1bis_Critical_hyper}. 
Furthermore, since Assertion (C4) holds, it follows from Lemma~\ref{lem_6_Critical_hyper} that 
\eqref{E1bis_thm1bis_Critical_hyper} is satisfied by each 
$C\in\mathfrak{C}_{{\rm even}}(\Gamma)$ such that $v(C)=2$. 
Similarly, since Assertion (C5) holds, it follows from Lemma~\ref{lem_7_Critical_hyper} that 
\eqref{E1bis_thm1bis_Critical_hyper} is satisfied by each 
$C\in\mathfrak{C}_{{\rm even}}(\Gamma)$ such that $v(C)=3$. 
It follows that \eqref{E1bis_thm1bis_Critical_hyper} is satisfied by every 
$C\in\mathfrak{C}(\Gamma)\setminus\mathfrak{C}_{1}(\Gamma)$. 
To prove that \eqref{E1_thm2_Critical_hyper} holds, it remains to prove that 
$d_{\mathscr{P}(\Gamma\bullet\mathbb{H})}(v)=0$ for each $v\in V_1(\Gamma)$. 
Set $\mathcal{N}=\{v\in V_1(\Gamma):d_{\mathscr{P}(\Gamma\bullet\mathbb{H})}(v)\neq 0\}$. 
For a contradiction, suppose that $\mathcal{N}\neq\emptyset$. 
Since \eqref{E1bis_thm1bis_Critical_hyper} is satisfied by every 
$C\in\mathfrak{C}(\Gamma)\setminus\mathfrak{C}_{1}(\Gamma)$, we get 
\begin{equation}\label{E3_thm2_Critical_hyper}
\text{for each $v\in\mathcal{N}$, 
$N_{\mathscr{P}(\Gamma\bullet\mathbb{H})}(v)\subseteq\mathcal{N}$.}
\end{equation}
For a contradiction, suppose that there exists $v\in\mathcal{N}$ such that 
$d_{\mathscr{P}(\Gamma\bullet\mathbb{H})}(v)=1$. 
By Lemma~\ref{lem_8_Critical_hyper}, 
there exists $w\in V_1(\Gamma)$ satisfying 
\begin{itemize}
\item $\{w\}$ is an isolated vertex of $\mathbb{H}-\{v\}$;
\item $\mathbb{H}-\{\{v\},\{w\}\}$ is connected;
\item for every nontrivial module $\mathbb{M}_{vw}$ of $\mathbb{H}-\{\{v\},\{w\}\}$, 
we have $\mathbb{M}_{vw}\setminus(\mathfrak{C}_{1}(\Gamma)\setminus\{\{v\},\{w\}\})\neq\emptyset$. 
\end{itemize}
Therefore, Assertion (C6) is not satisfied for $C=\{v\}$ and $D=\{w\}$. 
It follows from Lemma~\ref{Lcritical} that $d_{\mathscr{P}(\Gamma\bullet\mathbb{H})}(v)=2$ 
for every $v\in\mathcal{N}$. 
Moreover, by \eqref{E3_thm2_Critical_hyper}, $N_{\mathscr{P}(\Gamma\bullet\mathbb{H})}(v)\subseteq\mathcal{N}$ for every $v\in\mathcal{N}$. 
Therefore, $\mathscr{P}(\Gamma\bullet\mathbb{H})$ admits a component $\mathcal{C}$ such that 
$v(\mathcal{C})\geq 3$, 
$V(\mathcal{C})\subseteq\mathcal{N}$, and 
$\mathscr{P}(\Gamma\bullet\mathbb{H})[V(\mathcal{C})]$ is a cycle. 
Since $\mathfrak{C}(\Gamma)\setminus\mathfrak{C}_1(\Gamma)\neq\emptyset$ by \eqref{critical_construction_1}, we obtain 
$V(\mathcal{C})\subsetneq V(\Gamma\bullet\mathbb{H})$, which contradicts Proposition~\ref{Pcritical_T2n+1} because $\Gamma\bullet\mathbb{H}$ is critical. 
It follows that $\mathcal{N}=\emptyset$. 
Therefore, \eqref{E1_thm2_Critical_hyper} holds. 
\end{proof}

\end{document}